\documentclass[11pt]{article}
\usepackage[sort&compress,numbers,comma,square]{natbib}
\RequirePackage{ifthen}
\RequirePackage{amsmath, amsthm, amssymb, amsfonts, enumitem}
\RequirePackage[mathcal]{euscript}  
\RequirePackage[active]{srcltx}
\RequirePackage[usenames,dvips]{color}
\RequirePackage[colorlinks]{hyperref}

\long\def\comment#1{}

\newtheorem{theorem}{Theorem}[section]
\newtheorem{prop}[theorem]{Proposition}

\newtheorem{lemma}[theorem]{Lemma}

\theoremstyle{definition}
\newtheorem{example}[theorem]{Example}
\theoremstyle{remark}
\newtheorem{remark}{Remark}[section]
\numberwithin{equation}{section}

\def\Grp#1{\left(#1\right)}
\def\Cbr#1{\left\{#1\right\}}
\def\Sbr#1{\left[#1\right]}
\def\Flr#1{\left\lfloor#1\right\rfloor}
\def\Cil#1{\left\lceil#1\right\rceil}
\def\Abs#1{\left|#1\right|}
\def\cf#1{\mathbf{1}\!\Cbr{#1}}

\def\cum#1#2{#1_1+\cdots+#1_{#2}}		
\def\inum#1{#1_1, #1_2, \ldots}             	
\def\eno#1#2{#1_1, \ldots, #1_{#2}}         	

\def\gv{\,|\,}

\def\toi{\to\infty}

\def\Lsup{\varlimsup}

\def\Iff{\Longleftrightarrow}  	
\def\Lrif{\Longrightarrow}     	

\def\AND{,\,}
\def\cdf{F}  
\def\cbfrac#1#2{\{#1/#2\}}  
\def\cfx#1{\mathbf{1}\{#1\}}
\def\cvp#1{\sp{#1^*}}
\def\dbin{\mathrm{Bin}}
\def\dd{\mathrm{d}}
\def\dpois{\mathrm{Poisson}}
\def\Ints{\mathbb{Z}}
\def\intzi{\int_0^\infty}
\def\Iof{\mathcal{K}} 
\def\iof{K} 
\def\ldh{H} 
\def\levy{\text{L\'evy}}
\def\limdf{\phi_\cdf} 
\def\lfrac#1#2{#1/#2}  
\def\iunit{\mathrm{i}}
\def\mean{\mathbb{E}}
\def\Nats{\mathbb{N}}
\def\nth#1{\frac{1}{#1}}
\def\ppar{\varrho}
\def\pr#1{\mathbb{P}\{#1\}}

\def\Reals{\mathbb{R}}
\def\renf{U}  
\def\RV{\mathcal{R}}
\def\rvf{A}     
\def\rx{\epsilon}
\def\sumoi#1{\sum_{#1=1}^\infty}
\def\sumzi#1{\sum_{#1=0}^\infty}
\def\tail#1{\overline#1}
\def\tailr{r_\cdf} 

\def\xpp{p^+}   
\def\XR{R_{T,\eta,r,c_1,c_2}}
\def\Xr{R}

\begin{document}
\begin{center}
  \textbf{\MakeUppercase{
      Integral criteria for Strong Renewal Theorems with infinite mean
    }
  } \\[1ex]
  \normalsize
  Zhiyi Chi\\
  Department of Statistics\\
  University of Connecticut \\
  Storrs, CT 06269, USA \\[.5ex]
  E-mail: zhiyi.chi@uconn.edu
\end{center}
  
\begin{abstract}
  Let $\cdf$ be a probability measure on $\Reals$ in the domain of
  attraction of a stable law with exponent $\alpha\in (0,1)$.  We
  establish integral criteria on $\cdf$ that significantly expand the
  probabilistic approach to Strong Renewal Theorems (SRTs).  The
  criterion for $\alpha \in (0,1/2]$ is much weaker than currently
  available ones and in some cases provides sufficient and necessary
  conditions for the SRT.  The criterion for $\alpha \in (1/2, 1)$
  establishes the SRT in full generality and in a unified way, barring
  the Limit Local Theorems employed.  As an application, for
  infinitely divisible $\cdf$, an integral criterion on its $\levy$
  measure is established for the SRT.  As another application, for
  $\cdf$ in the domain of attraction of a stable law without
  centering, an integral criterion on $\cdf$ is established for the
  SRT for the ladder height process of a random walk with step
  distribution $\cdf$.

  \paragraph{\it Keywords and phrases.} Renewal, regular variation,
  ladder height process, infinitely divisible, large deviations.
  
  \paragraph{\it 2010 Mathematics Subject Classification.} 60K05,
  60F10.
\end{abstract}

\section{Introduction} \label{s:intro}
Let $\cdf$ be a probability measure on $\Reals$ and $\cdf\cvp n$ its
$n$-fold self-convolution, with $\cdf\cvp 0$ being
the unit mass at 0.  The renewal measure associated with $\cdf$ is
\begin{align} \label{e:renew-measure} 
  \renf(\dd x):=\sumzi n \cdf\cvp n(\dd x).
\end{align}
This paper concerns the strong renewal theorem (SRT) for $\cdf$, i.e.\
the nontrivial asymptotic of $\renf((x, x+h])$ as $x\toi$, with $h\in
(0,\infty)$, when $\cdf$ has infinite mean and is in the domain
of attraction of a stable law.  More precisely, denoting $\cdf(x) =
\cdf((-\infty, x])$, suppose $\cdf$ satisfies the 
following two tail conditions.
\begin{enumerate}[itemsep=0ex]
\item Regular right tail: for some $\alpha \in (0,1)$ and $\rvf \in
  \RV_\alpha$,
  \begin{align}  \label{e:tail-RV}
    \text{$\tail \cdf(x) := 1 - \cdf(x)
      \sim \lfrac{1}{\rvf(x)}$ as $x\toi$}.
  \end{align}
\item Converging tail ratio:
  \begin{align} \label{e:tail-ratio}
    \tailr :=\lim_{x\toi} \cbfrac{\cdf(-x)}{\tail\cdf(x)}
    \ \text{ exists and is finite}.
  \end{align}
\end{enumerate}
In the regular right tail condition, $\RV_\alpha$ stands for the
class of functions that are \emph{strictly positive\/} on $(0,\infty)$
and regularly varying at $\infty$ with exponent $\alpha$.  Denote
$\RV = \cup_\alpha \RV_\alpha$.  In this paper, $\cdf$ is said to be
arithmetic if it is concentrated on $h\Ints$ for some $h>0$, and said
to be lattice if it is concentrated on $a+h\Ints$ for some $0\le a<h$.
In either case, the span of $\cdf$ is the largest $h$ with the above
property.  A lattice distribution on $a+h\Ints$ is non-arithmetic if
and only if $a/h$ is irrational.  Henceforth, $\cdf$, $\rvf$, and $h$
are fixed, and $I\equiv (0,h]$.  Unless mentioned otherwise, if $\cdf$
is non-arithmetic, $h$ can be any positive number; however, if $\cdf$
is arithmetic, $h$ is its span.  Denote $\xpp = \pr{X>0} =
\tail\cdf(0)$.  By \eqref{e:tail-RV}, $\xpp>0$.

It is a classical result that if $\alpha \in (1/2,1)$, then $x\tail
\cdf(x) \renf(x+I)\to C h$ with $C>0$ a constant, when $\cdf$ is
arithmetic \cite{garsia:lamperti:63, williamson:68}, or is
non-arithmetic and concentrated on $[0,\infty)$ \cite
{erickson:70:tams}.  It is also known that if $\alpha \in (0,1/2]$,
then the SRT in general does not hold without extra conditions
\cite{williamson:68}.  There are two main approaches to the SRT, one
being Fourier analytic, the other probabilistic.  When 
$\xpp=1$, by directly attacking the Fourier transform of $\renf(x+I)$,
the SRT can be established \cite {garsia:lamperti:63,
  erickson:70:tams}.  However, this approach critically depends on the
assumption $\alpha\in (1/2,1)$ so it cannot be extended to $\alpha \le
1/2$.  The probabilistic approach basically separates $\renf(x+I)$
into two parts, one being the sum of $\cdf\cvp n(x+I)$ over $n\ge
\rvf(\delta x)$ with $\delta>0$ being arbitrarily small, and the other
being the sum over $n<\rvf(\delta x)$, which we will refer to as the
``small-$n$ contribution''.  Since the convergence of the 
first part can be resolved using Local Limit Theorems (LLTs) and
Riemann sum approximation \cite{garsia:lamperti:63, erickson:70:tams,
  chi:14:aap}, essentially, the SRT holds if and only if the small-$n$
contribution vanishes as $\delta\to 0$ \cite{chi:14:aap}.  For
arithmetic $\cdf$, if $\alpha\in (1/2,1)$, or if $\alpha\in (1/4,1/2]$
and
\begin{gather} \label{e:Williamson}
  \sup_{x>0} \omega(x) < \infty,
  \\
  \text{where}\ \
  \omega(x) = \lfrac{x \cdf(x+I)}{\tail\cdf(x)},
  \label{e:prob-ratio}
\end{gather}
then by Fourier analysis, the vanishing small-$n$ contribution can be
established \cite{williamson:68}.  However, for $\alpha\in (0,1/4]$,
Fourier analysis failed to work, and it took quite a while until
\cite{doney:97} to resolve the issue for arithmetic $\cdf$ satisfying
\eqref{e:Williamson}.  A key ingredient of the proof in
\cite{doney:97} was a local large deviation (LLD) estimate.  Recently,
for non-lattice $\cdf$ satisfying \eqref{e:Williamson}, using a rather
simple argument that bypassed the LLD estimate, the SRT was
established \cite{vatutin:13:tpa}.  Nevertheless, the LLD estimate
turns out to be useful in general.  It has been refined by the
analysis of ``small-step sequences'' in the study on LLDs of sums of
random variables \cite{denisov:08:ap}.  A similar estimate will play
an important role in this paper.

One problem with the uniform bound condition \eqref{e:Williamson}
is that it is quite restrictive.  For example, one can easily find
examples of arithmetic $\cdf$ with $\alpha \in (1/2, 1)$ that fail to
satisfy \eqref{e:Williamson}, even though it is a foregone
conclusion that the SRT holds for $\cdf$.  Therefore, to expand the
scope of the probabilistic approach, a much weaker condition is
needed.

In \cite{chi:14:aap}, it is shown that integrals of large values of
$\omega$ can provide much weaker sufficient conditions for the SRT.  
This paper further pursues the idea.  Define for $\eta\in (0,1]$,
\begin{align} \label{e:overflow}
  \iof_\eta(x,T)
  =
  \int_{(1-\eta)x} ^x \Sbr{\omega(y) - T}^+\dd y,
\end{align}
where $c^\pm = (\pm c)\vee 0$ for $c\in\Reals$.  The special case
$\iof_1(x,T)$ was studied in \cite{chi:14:aap}.  However, conditions
based on $\iof_1(x,T)$ are still too strong, for example, to provide a
unified treatment of the small-$n$ contribution when $\alpha \in (1/2,
1)$.  This limitation disappears with $\eta<1$.

The main technical tool of the paper is given in Theorem
\ref{t:sum-small-n}, which is a bound for the small-$n$
contribution expressed in terms of integrals involving $\omega$.
The bound is tighter than the one provided by $\iof_\eta$ and is of
interest in its own right.  On the other hand, $\iof_\eta$ is more
convenient to use.  From Theorem \ref{t:sum-small-n} several SRTs
based on $\iof_\eta$ will be derived.  Comparing to the results in
\cite{chi:14:aap}, these SRTs require weaker conditions and often
have much shorter proofs.  To state these SRTs, first some
notation.  Let $S_n = \cum X n$ with $X_i$ i.i.d.\ $\sim \cdf$.
Denote $\rvf^-(x) = \inf\{t\ge 1: \rvf(t)> x\}$, where the restriction
that $t\ge 1$ is to avoid non-essential problems in case $\rvf(0+) >
x$.  Since $\rvf$ tends to $\infty$ at $\infty$, $\rvf^-$ is
well-defined on $(0, \infty)$ and $\rvf^-\in \RV_{1/\alpha}$
\cite[][Th.~1.5.12] {bingham:89}.  Denote $a_n = \rvf^-(n)$.  By
definition, $a_0=1$.  Under \eqref{e:tail-RV} -- \eqref{e:tail-ratio},
$S_n / a_n$ converges in distribution to a stable random variable
$\zeta$ of exponent $\alpha$ \cite[][p.~207--213]{breiman:92}.
Henceforth, denote
\begin{align*}
  \limdf = \text{density of $\zeta$}, \quad
  \ppar = \ppar_\cdf = \pr{\zeta>0}.
\end{align*}
Then $\ppar = 1/2 + (\pi\alpha)^{-1} \tan^{-1}(c\tan(\pi \alpha/2))$
with $c = (1-\tailr) / (1+\tailr)$ \cite[][p.~380]{bingham:89}.  Our
first result concerns the case $\alpha \in (0, 1/2]$.  Comparing to
Theorem 2.1 of \cite{chi:14:aap}, it requires weaker conditions and as
the same time has more clear conclusions.
\begin{theorem} \label{t:SRT-alpha-LE-1/2}
  Let $\alpha\in (0,1/2]$ and \eqref{e:tail-RV} --
  \eqref{e:tail-ratio} hold.  Suppose for some $L\in\RV$
  bounded below by 1 and constants $T\ge 0$,
  $\eta \in (0,1]$, the following are true as $x\toi$.
  \begin{enumerate}[topsep=1ex, itemsep=0ex, leftmargin=4ex,
    label={\alph*})]
  \item $L(x)/\rvf(x)\to 0$ and
    \begin{gather} \label{e:low-cut}
      \frac{x}{\rvf(x)} \sum_{n<L(x)} \cdf\cvp n(x+I)\to 0.
    \end{gather}
  \item If $\alpha\in (0,1/2)$, then
    \begin{align} \label{e:diff2}
      &
      \iof_\eta(x, T)
      =
      o\Grp{
        \lfrac{\rvf(x)^2 \rvf^-(L(x))}{L(x)^2}
      }.
    \end{align}
  \item If $\alpha=1/2$, then letting $u(x) = \int_1^x
    \cbfrac{\rvf(s)}{s}^2\,\dd s$ and $k(x) = \rvf(x)^2 / u(x)$,
    \begin{align}  \label{e:diff}
      &
      \iof_\eta(x, T)
      =
      \begin{cases}
        \displaystyle
        O(k(x)), &
        \displaystyle
        \text{if}\ 
        \lfrac{u(x)}{u(\rvf^-(L(x)))}\to 1
        \\[1.5ex]
        o(k(x)), & \text{else.}
      \end{cases}
    \end{align}
  \end{enumerate}
  Then the SRT holds, i.e.
  \begin{align} \label{e:renewal}
    \lim_{x\toi} x\tail\cdf(x) \renf(x+I) =
    \alpha h\intzi x^{-\alpha} \limdf(x)\,\dd x.
  \end{align}
\end{theorem}

\begin{remark}
  If \eqref{e:low-cut} -- \eqref{e:diff} hold for one $h>0$, they
  hold for all $0<h<\infty$.  If $\cdf$ satisfies
  \eqref{e:Williamson}, then for any $T>\sup\omega$ and $\eta\in
  (0,1]$, $\iof_\eta(x, T)\equiv 0$, and so one can set
  $L(x)\equiv 1$.
\end{remark}
\begin{remark} \label{r:small-n-truncation}
  From the Appendix in \cite{chi:14:aap}, the SRT holds $\Iff$
  \begin{align} \label{e:small-n-truncation}
    \lim_{\delta\to 0}
    \Lsup_{x\toi} x\tail \cdf(x)
    \sum_{n<\rvf(\delta x)} 
    \cdf\cvp n(x+I)=0.
  \end{align}
  In particular,
  \begin{align} \label{e:SRT-omega}
    \text{SRT} \Lrif \omega(x) = o(\rvf(x)^2).
  \end{align}
  This may give a hint why $\rvf(x)^2$ appears in the bounds in
  \eqref{e:diff2} and \eqref{e:diff}.
\end{remark}

\begin{example} \label{ex:Williamson}
  Consider the following example adapted from \cite{williamson:68}.
  Let $\cdf$ be concentrated on $\Ints\setminus \{0\}$, such that
  for $n\ne 0$,
  \begin{align*}
    \cdf\{n\}
    =
    \begin{cases}
      C 2^{-k/2} / b^\pm_k
      & \text{$n=\pm 2^k$ for some $k\in\Nats$}
      \\
      C |n|^{-3/2} \ln |n| & \text{otherwise}
    \end{cases}
  \end{align*}
  where $\inf_k b^\pm_k>0$ and $C>0$ is the normalizing constant.  In
  \cite{williamson:68}, $b^+_k = b^-_k=k$.  It is easy to show that
  $\tail\cdf(x) \sim\cdf(-x) \sim D |x|^{-1/2} \ln |x|$ as $x\toi$
  for a constant $D\in (0,\infty)$.  Therefore, $\cdf$ satisfies
  the tail conditions \eqref{e:tail-RV} -- \eqref{e:tail-ratio} and
  one can set $\rvf(x) = D^{-1}\sqrt{x}/\ln x$ for $x\ge 2$.  However,
  if $\Lsup_k b^+_k/k<\infty$, then there is a constant $C>0$, such
  that for infinitely many $n=2^k$, $b^+_k \le C k$.  For these
  $n=2^k$,
  \begin{align*}
    \cbfrac{n}{\rvf(n)} \cdf\{n\}
    \gg \frac{2^k}{2^{k/2}/k} \frac{2^{-k/2}}{k} \equiv 1,
  \end{align*}
  so by \eqref{e:SRT-omega} the SRT cannot hold.  Here $u\asymp v$
  means that functions $u$ and $v$ defined on the same domain satisfy
  both $u\ll v$ and $u\gg v$, where $u\ll v$ stands for $|u|\le C |v|$
  for some constant $C$ and $u\gg v$ is the same as $v\ll u$.
  
  We show that if instead $b^+_k/k\toi$ as $k\toi$, then the SRT
  holds.  First $h=1$.  Let $L(x)\equiv 1$.  Then \eqref{e:low-cut}
  trivially holds.  Since $\alpha=1/2$, we need to check
  \eqref{e:diff}.   Each $(x, x+1]$ contains exactly one integer, say
  $n$.  Then $\cdf(x+I) = \cdf\{n\}$.  If $n\not\in \{2^k,
  k\in\Nats\}$, then $\omega(x) = \lfrac{x\cdf\{n\}}{\tail\cdf(x)}
  \asymp 1$.  If $n= 2^k$ for some $k\in\Nats$, then by $\ln x \asymp
  k = o(b^+_k)$,
  \begin{align*}
    \omega(x)
    = \frac{C x n^{-1/2}}{b^+_k \tail\cdf(x)}
    \asymp
    \frac{C x n^{-1/2}}{D b^+_k x^{-1/2} \ln x}
    =
    O\Grp{\frac{x}{b^+_k\ln x}}
    =
    o(\rvf(x)^2).
  \end{align*}
  Fix $\eta\in (0, 1/2)$.  For $x>0$, let $J_x = \{y\in [(1-\eta) x,
  x]: 2^k\in (y,y+1]$ for some $k\in \Nats\}$.   For $x$ large enough,
  $J_x$ is either empty or a single interval of length at most 1.  For
  $y\in [(1-\eta) x, x]\setminus J_x$, $\omega(y) = O(1)$, while for
  $y\in J_x$, $\omega(y) = o(\rvf(x)^2)$.  Consequently, given $T>0$
  large enough, $\iof_\eta(x,T)  = \int_{J_x} [\omega(y)-T]^+\,\dd y =
  o(\rvf(x)^2)$.  On the other hand, by $\cbfrac{\rvf(x)}{x}^2  \sim
  x^{-1} (\ln x)^{-2}$, $u(\infty-)<\infty$.  Therefore,
  \eqref{e:diff} holds, completing the proof.

  Finally, if $\cdf\{n\} = C|n|^{-1/2}/b^+_k$ for $n =
  \lceil c^k\rceil$ instead of $2^k$, where $c>1$ is fixed,
  then the above argument still works by letting $\eta \in (0,1-1/c)$.
  \qed
\end{example}

We next consider the case $\alpha \in (1/2,1)$.  Currently available
proofs for this case heavily rely on Fourier analysis and are
substantially different for the arithmetic and non-arithmetic $\cdf$
\cite{garsia:lamperti:63, williamson:68, erickson:70:tams}.  Theorem
\ref{t:sum-small-n} provides a unified proof for the small-$n$
contribution, whether $\cdf$ is arithmetic or non-arithmetic.  In
addition, the proof applies to non-arithmetic $\cdf$ supported by the
entire $\Reals$, for which the SRT appears yet to be established in
the literature. 
\begin{theorem} \label{t:SRT-alpha-GT-1/2} 
  Let  $\alpha\in (1/2,1)$ and \eqref{e:tail-RV} -- \eqref
  {e:tail-ratio} hold.  Then the SRT holds.
\end{theorem}

Now let $\cdf$ be infinitely divisible with $\levy$ measure $\nu$.
Since $\nu$ is usually much easier to specify than $\cdf$, a question
is whether conditions similar to those on $\iof_\eta(x, T)$ are
available for $\nu$ to guarantee the SRT.  For any $r>0$,
$\nu_r(\cdot) = \nu(\cdot\setminus (-r,r))$ is a finite measure.
Without loss of generality, assume $\nu_1\ne 0$.  Define distribution
function $\cdf_\nu(x) = \nu_1((-\infty, x]) / \nu_1(\Reals)$, and
$\omega_\nu(x)$ and $\iof_{\eta, \nu}(x,T)$ in terms of $\cdf_\nu$
according to \eqref{e:prob-ratio} and \eqref{e:overflow},
respectively.  The next result refines Theorem 2.5 of
\cite{chi:14:aap} and has a much shorter proof.

\begin{theorem} \label{t:renewal-inf-div}
  Let $\alpha\in (0,1/2]$ and $\nu$ satisfy
  \begin{align} \label{e:tail-ind}
    \rvf_\nu(x):=\nth{\nu((x,\infty))} \in \RV_\alpha, \quad x>0
  \end{align}
  and
  \begin{align} \label{e:tail-ratio-ind}
    \tailr :=\lim_{x\toi} \frac{\nu((-\infty, -x])}{\nu((x,\infty))}
    \ \text{ exists and is finite}.
  \end{align}
  Let $L$, $T$ and $\eta$ be as in Theorem  \ref{t:SRT-alpha-LE-1/2}.
  Suppose that instead of \eqref{e:low-cut}, $\cdf_\nu$ satisfies
  \begin{gather} \label{e:low-cut-inf-div}
    \lim_{\rx\to 0} \Lsup_{x\toi}
    \frac{x}{\rvf_\nu(x)} \sum_{n<L(x)} \sup_{|t-x|\le \rx x}
    \cdf_\nu\cvp n(t+I)= 0
  \end{gather}
  as well as \eqref{e:diff2} -- \eqref{e:diff} with all the quantities
  therein replaced with the corresponding ones defined in terms of
  $\cdf_\nu$.  Then, letting $\rvf(x) = \rvf_\nu(x)$ for $x>0$, the
  SRT \eqref{e:renewal} holds for $\cdf$.
\end{theorem}

Both Theorems \ref{t:SRT-alpha-LE-1/2} and \ref{t:renewal-inf-div}
assume the existence of a function $L$ that acts as a cut-off for $n$,
so that the contribution to $\renf(x+I)$ from $n<L(x)$ can be ignored
from the beginning.  It is known that if the SRT holds, then
\eqref{e:low-cut} holds for any $L(x) = o(\rvf(x))$ \cite{chi:14:aap}.
The question is what $L$ can be \emph{a priori\/}, i.e.\ before
knowing whether or not the SRT holds.  By \eqref{e:diff2} and
\eqref{e:diff}, the faster $L$ grows, the weaker the assumptions on
$\iof_\eta(x,T)$ are.  The following result provides a prior lower
bound on the growth of $L$.  Comparing to Proposition 2.3 in
\cite{chi:14:aap}, the result is substantially refined.  In view of
\eqref{e:SRT-omega}, the condition \eqref{e:tail-M} below is nearly
minimal, except for the constraint on $M$.  Also, the convergence in
\eqref{e:low-cut-uniform} is a little stronger than \eqref{e:low-cut}
and \eqref{e:low-cut-inf-div}.
\begin{prop} \label{p:L-prior}
  Let $\alpha \in (0,1/2]$ and $\cdf$ satisfy \eqref{e:tail-RV}.
  Suppose there is a non-decreasing function $M\in \RV_\beta$ with
  $\beta\in [1 - 2\alpha,1]$, such that as $x\toi$,
  \begin{align} \label{e:tail-M}
    \omega(x) \ll x/M(x) = o(\rvf(x)^2).
  \end{align}
  Let $g(x) = \rvf(x)\sqrt{M(x)/x}$.  Then the following are true.
  \begin{enumerate}[topsep=1ex, itemsep=0ex, leftmargin=4ex,
    label={\arabic*})]
  \item  If there is $\gamma>\alpha+\beta$ such that
    $g(x)\gg (\ln x)^\gamma$, then for any $\theta>0$,
    \begin{align}  \label{e:low-cut-uniform}
      \Lsup_{x\toi}
      \frac{x}{\rvf(x)} \sum_{n<L(x)} \sup_{t\ge \theta x}
      \cdf\cvp n(t+I)= 0
    \end{align}
    holds for any non-decreasing $L\in\RV$ that satisfies $L(x) \ll
    g(x)/(\ln x)^\gamma$.
  \item Given $\rx\in (0,1/(1+\alpha+\beta))$, \eqref
    {e:low-cut-uniform} holds for any non-decreasing $L\in\RV$ that
    satisfies $1\ll L(x) \ll g(x)^\rx$.
  \end{enumerate}
\end{prop}

\begin{example} \label{ex:Density}
  Suppose $\alpha \in (0,1/2)$ and $\omega(x) = O(x^c)$ with $c\in [0,
  2\alpha)$.  Let $M(x) = x^{1-c}$.  By 1) of
  Proposition \ref{p:L-prior}, \eqref{e:low-cut-uniform} holds if
  $L(x) = x^p$ with $p \in [0, \alpha - c/2)$.  Then by Theorem
  \ref{t:SRT-alpha-LE-1/2}, the SRT holds if $\iof_\eta(x,T)=O(x^q)$
  with $q\in [0, 1- c/(2\alpha) + c)$, which implies that on
  $[(1-\eta) x, x]$, the average density of $t$ with $\omega(t) \asymp
  t^c$ can be as high as $(\iof_\eta(x, T)/x^c)/x  \gg
  x^{-c'/2\alpha}$ for any $c'>c$.   This may be compared to Example
  \ref{ex:Williamson}, where the density of $t$ with large $\omega(t)$
  has an exponential decay.  In contrast, 2) of Proposition \ref
  {p:L-prior} only guarantees that \eqref{e:low-cut-uniform} holds for
  $L(x) = x^p$ with $p < (\alpha-c/2)/(2 +  \alpha -c)$.

  On the other hand, if $\omega(x)$ grows almost as fast as
  $\rvf(x)^2$, for example, $\rvf(x)^2/\omega(x) = O(\ln\ln x)$, then
  $g(x) \ll \rvf(x)/ \sqrt{\omega(x)}= O(\sqrt{\ln\ln x})$ and 1) is
  not applicable.  However, since one can set $M(x) \sim x\ln\ln
  x/\rvf(x)^2 \in\RV_{1-2\alpha}$, by 2), \eqref{e:low-cut-uniform}
  holds for $L(x)\asymp (\ln \ln x)^q$ with $q\in [0,
  1/(4-2\alpha))$.
  \qed
\end{example}

In the above example, the asymptotic density of locations with large
values of $\omega$ is obtained with no assumption on how these
locations are distributed on $\Reals$.  It turns out that if the
locations are distributed more regularly, then their asymptotic
density can be higher while still allowing the SRT to hold.
Specifically, given $0<c\le s<\infty$, a set $E$ is said to
(asymptotically) have density $O(x^{-c})$ at scale $x^s$, if for all
$x\gg 1$ and $y\ge x^s$, $|E\cap (x, x+y)| = O(x^{-c} y)$, where
$|\cdot|$ denotes the Lebesgue measure.  For example, given $c\in
(0,1)$, the union of $n^{1/(1-c)}+(0,1)$, $n\ge 1$, has density
$O(x^{-c})$ at scale $x^c$.  Then we have the next result.

\begin{prop} \label{p:density}
  Let $\alpha\in (0,1/2]$ and $c\in (0, 2\alpha)$.  Suppose $\omega(x)
  = O(x^c)$.  If for some $T\in [0,\infty)$ and $c\le s<2\alpha$,
  $E_T:=\{x\ge 0: \omega(x)>T\}$ has density $O(x^{-c})$ at scale
  $x^s$, then the SRT holds.
\end{prop}

Now we consider the SRT for ladder height processes; here $\alpha$ may
be in $[1,2]$.  Denote by $\ldh_n$ the (strict) ascending 
ladder height process of $S_n$.  Since $\cdf$ is the basic
information, it is desirable to find conditions on $\cdf$ that yield
the SRT for $\ldh_n$.  A related issue, the SRT for the ladder time
process, is resolved in \cite{doney:97}.  Denote by $\cdf_+$ the
distribution of $\ldh_1$.  Suppose that under conditions
\eqref{e:tail-RV} and \eqref{e:tail-ratio}, $\cdf$ is in the domain of 
attraction of stable law \emph{without centering\/}.  Then, as
$x\toi$, $\tail \cdf\!_+(x) \sim 1/\rvf_+(x)$ with $\rvf\!_+\in
\RV_{\alpha\ppar}$; actually the asymptotic of $\rvf_+(x)$ can be
obtained, although somewhat implicitly in most cases \cite{rogozin:71,
  greenwood:82b:zw, doney:93:ptrf}.  Denote by $\renf_+$ the renewal
measure for $\ldh_n$.  Define the \emph{weak\/} descending ladder
process of $S_n$ as the weak ascending ladder process of $-S_n$.
Denote by $\cdf\!_-$ and $\renf_-$ the corresponding step distribution
and renewal measure.  Since the ladder steps are non-negative and
$\ldh_n = S_n$ if $\xpp=1$, from \cite{garsia:lamperti:63,
  erickson:70:tams}, one only need consider the case where $\alpha
\ppar\le 1/2$ and $\xpp\in (0,1)$.  In this case, the
maximum possible value of $\alpha$ is 3/2.  The next result refines
Theorem 2.4 of \cite{chi:14:aap} and has a much shorter proof.

\begin{theorem} \label{t:ladder}
  Let $\alpha\in (0,3/2]$ and \eqref{e:tail-RV} --
  \eqref{e:tail-ratio} hold with $\xpp\in (0,1)$, such that $S_n/a_n$
  converges in distribution to a non-zero stable random variable and
  $\alpha\ppar\in (0,1/2]$.  Suppose for some $L\in\RV$ bounded below
  by 1 and constants $T\ge 0$, $\eta\in (0,1]$, the following are true
  as $x\toi$.
  \begin{enumerate}[topsep=1ex, itemsep=0ex, leftmargin=4ex,
    label={\alph*})]
  \item $\tail\cdf\!_+(x) L(x)\to 0$ and 
    \begin{gather} \label{e:low-cut-ladder}
      x\tail\cdf\!_+(x) \sum_{n<L(x)} \pr{\ldh_n\in x+I}\to 0,
    \end{gather}
  \item If $\alpha\ppar\in (0,1/2)$, then
    \begin{align} \label{e:diff2-ladder}
      \iof_\eta(x,T) = o\Grp{
        \lfrac{\rvf_+(x)^2 \rvf_+^-(L(x))}{L(x)^2}
      }
    \end{align}
  \item If  $\alpha\ppar = 1/2$, then letting $u_+(x) = \int_1^x
    \cbfrac{\rvf_+(s)}{s}^2\,\dd s$ and $k_+(x) = \rvf_+(x)^2 /
    u_+(x)$,
    \begin{align}  \label{e:diff2-ladder-2}
      &
      \iof_\eta(x, T)
      =
      \begin{cases}
        \displaystyle
        O(k_+(x)), &
        \displaystyle
        \text{if}\ 
        \lfrac{u_+(x)}{u_+(\rvf_+^-(L(x)))}\to 1
        \\[1.5ex]
        o(k_+(x)), & \text{else.}
      \end{cases}
    \end{align}
  \end{enumerate}
  \comment{
    \begin{gather}
      \iof_\eta(x, T)
      \ll
      \Grp{\rvf(x)^{2\ppar} L(x)^{1/(\alpha\ppar) - 2}}^c.
      \label{e:diff2-ladder}
    \end{gather}
  }
  Then the SRT holds for $\cdf_+$, i.e.\  $x \tail\cdf\!_+(x)
  \renf_+(x+I) \to h \sin(\pi\alpha\ppar)/\pi$ as $x\toi$.
\end{theorem}

\begin{remark}
  Under the same conditions, the SRT also holds for the weak ladder
  process. \qed
\end{remark}

Condition \eqref{e:low-cut-ladder} involves the distributions of
$\ldh_n$ that in general are unknown.  Based on Proposition
\ref{p:L-prior}, the next result provides a sufficient condition on
$\cdf$ so that \eqref{e:low-cut-ladder} holds.
\begin{prop} \label{p:L-prior-ladder}
  Suppose there exist $c\in (0,1)$ and a non-decreasing function $M\in
  \RV_\beta$ with $\beta\in [1-2c\alpha\ppar, 1]$, such that for 
  $x\ge 1$, $\omega(x) \ll x/M(x)$.  Let $g(x) = x^{2c\alpha\ppar}
  \sqrt{M(x)/x}$.  Then the following are true.
  \begin{enumerate}[topsep=1ex, itemsep=0ex, leftmargin=4ex,
    label={\arabic*})]
  \item  If there is $\gamma>\alpha\ppar+\beta$ such that
    $g(x)\gg (\ln x)^\gamma$, then \eqref{e:low-cut-ladder} holds for
    any non-decreasing $L\in\RV$ that satisfies $L(x) \ll g(x)/(\ln
    x)^\gamma$.
  \item In any case, \eqref{e:low-cut-ladder} holds if $1\ll L(x) \ll
    g(x)^\rx$, where $\rx$ is an arbitrary number in $(0,
    1/(1+\alpha\ppar+\beta))$.
  \end{enumerate}
\end{prop}

In the rest of the paper, Section \ref{s:tool} states bounds for the
small-$n$ contribution.  Section \ref{s:SRT-proofs} proves all the
theorems stated in this section and Section \ref{s:small-n}
establishes the main technical tool for the proofs.  Finally, Section
\ref{s:Prior} proves the propositions stated in this section.

\section{Bounds for small-\emph{n} contribution} \label{s:tool}
Since $\cdf$, $\rvf$ and $h$ are fixed, we expand the meaning of
several asymptotic notations as follows.  If $f$ and $g$ are functions
defined on some set $D$, then $f=O(g)$, $f\ll g$, $g=\Omega(f)$, and
$g\gg f$ all mean $|f(x)| \le C |g(x)|$ for $x\in D$, where $C$ is a
positive constant only depending on $\{\cdf, \rvf, h\}$, and $f =
\Theta(g)$ and $f\asymp g$ both mean $g\ll f\ll g$.  If, for example,
$|f|\le  C|g|$, where $C$ is a constant that depends on parameters
$\eno a n$ in addition to $\{\cdf, \rvf, h\}$, then denote $f=O_{\eno
  a n}(g)$ or $f\ll_{\eno a n} g$.  On the other hand, by $f(x) =
o(g(x))$ as $x\toi$ we mean there is a function $M(\rx)$ that only
depends on $\{\cdf, \rvf, h\}$ such that $|f(x)|\le \rx |g(x)|$ for
all $x\ge M(\rx)$.  Finally, by $f(x)\sim g(x)$ we mean $f(x) =
[1+o(1)] g(x)$.

We will assume without loss of generality that $\rvf$ is continuously
differentiable and strictly increasing on $(0,\infty)$, such that
$\rvf(1)=1$.   Under the assumption, $\rvf^-$ is the regular inverse 
$\rvf^{-1}$ and is continuous and strictly increasing on $[1,
\infty)$.  In particular $a_1 = \rvf^{-1}(1) = 1$.  Still, $a_0=1$.
The following facts will be often used,
\begin{align}
  \rvf^{-1}(t) \asymp a_{\lfloor t\rfloor}
  \asymp a_{\lceil t\rceil}, \quad t\ge 1,
  \label{e:an-inverse-rvf}\\
  \rvf'(s) \sim \alpha \rvf(s)/s, \quad s\toi.
  \label{e:rvf-diff}
\end{align}

Let $X$, $X_n$, $n\ge 1$, be i.i.d.\ $\sim \cdf$, and
$S_0=0$ and $S_n = \cum X  n$, $n\ge 1$.  Denote
\begin{align*}
  S^\pm_n = \sum_{i=1}^n X^\pm_i,\quad
  N_n = \sum_{i=1}^n \cf{X_i>0}.
\end{align*}
Note that $S^\pm_n\ne(\pm S_n)\vee 0$.  Still, $S_n = S^+_n - S^-_n$.
Also, $N_n\sim \dbin(n, \xpp)$, the binomial distribution with
parameters $n$ and $\xpp$.

Henceforth, define
\begin{align*}
  \kappa = \Flr{\lfrac{1}{\alpha}}
\end{align*}
and for $\eta\in (0,1]$, $n\ge 1$, $r>0$,
\begin{align} \label{e:IOF-def}
  \Iof_{\eta,n,r}(x,T)
  =
  \int_{(1-\eta)x}^x  e^{-(x-y)/(r a_n)}
  [\omega(y)- T]^+\,\dd y.
\end{align}

The main tool for the proofs of the SRTs is the following, which only
requires the regular right tail condition for $\cdf$.

\begin{theorem} \label{t:sum-small-n}
  Let $\alpha \in (0,1)$ and $\cdf$ satisfy \eqref{e:tail-RV}.  Then
  for all $L(x)>0$, $T\ge 0$, $\eta\in (0,1)$, $r\in (0,1]$, $1/2\le
  c_1<c_0\le 1\le c_2$, and $0<\delta\ll_{\eta,r} 1$,
  \begin{align*}
    &
    \Lsup_{x\toi}
    \frac{x}{\rvf(x)} \sum_{L(x)\le n <\rvf(\delta x)}
    \sup_{c_0 x\le y\le c_2 x} \cdf\cvp n(y + I)
    \\
    &\le 
    \Lsup_{x\toi} \XR(x, \delta)
    + O(\delta^{2\alpha})T + o(\delta^\alpha),
  \end{align*}
  where $o(\delta^\alpha)$ is in the sense of $\delta\to 0+$ and,
  writing $x_n = x+S^-_n$,
  \begin{align} \label{e:sum-small-n-R}
    &
    \XR(x,\delta)
    =
    \frac{x}{\rvf(x)} \sum_{L(x)\le n <\rvf(\delta x)}
    \mean\Sbr{
      \frac{N_n}{a_{N_n}}
      \frac{\tail\cdf(x_n)}{x_n}
      \sup_{c_1 x_n/\kappa \le t \le c_2 x_n + 2h}
      \Iof_{\eta,N_n,r}(t, T)
    }.
  \end{align}
\end{theorem}
\begin{remark}
  In the bound, the term $O(\delta^{2\alpha}) T$ is separate from
  $o(\delta^\alpha)$ because the latter is independent of $T$.
\end{remark}

On the other hand, integrals involving $\omega$ can also provide a
lower bound for the small-$n$ contribution as follows.
\begin{prop} \label{p:sum-small-n-low}
  Let $\alpha\in(0,1)$ and $J = (0,2h]$.  Then given compact interval
  $E$ with $\inf_E \limdf>0$, for all $0<\delta\ll_E 1$,
  $n_0\gg_E 1$, and $x \gg_E 1$,
  \begin{align*}
    \frac{x}{\rvf(x)} \sum_{n <\rvf(\delta x)}
    \cdf\cvp n(x + J)
    \ge
    \frac{1}{4\rvf(x)^2} \sum_{n_0\le n<\rvf(\delta x)}
    n\int_E
    \limdf(t) \omega(x-a_n t)\,\dd t,
  \end{align*}
  where $h>0$ is arbitrary if $\cdf$ is non-lattice, and the span of
  $\cdf$ otherwise.
\end{prop}

The next example shows an application of the lower bound.
\begin{example} \rm \label{ex:Williamson-2}
  Suppose the $\cdf$ in Example \ref{ex:Williamson} is modified as
  \begin{align*}
    \cdf\{n\}
    =
    \begin{cases}
      C 2^{-k/2} / b^\pm_k
      & \text{$n=\pm 2^k$ for some $k\in\Nats$}
      \\
      C |n|^{-3/2} g(n) & \text{otherwise}
    \end{cases}
  \end{align*}
  where $g\in \RV_0$ and $g(-x)/g(x)\to \tailr\in [0,\infty)$ as
  $x\toi$.  Following the argument in Example \ref{ex:Williamson}, it
  can be seen that for $n\ge 2$ and $x\in [n-1, n)$, $\rvf(x) \asymp
  \sqrt{x}/g(x)$ and
  \begin{align*}
    \omega(x)
    \asymp
    \begin{cases}
      1 & n\ne 2^k \\
      x/(b_k g(x)) & \text{otherwise}
    \end{cases}
  \end{align*}
  and if $u(\infty-)<\infty$, where $u(x) = \int_1^x t^{-1} g(t)^{-2}
  \,\dd t$, then the SRT holds $\Iff g(2^k) = o(b_k)$ as $k\toi$.  Of
  interest is the case where $u(\infty-)=\infty$.  By Theorem
  \ref{t:SRT-alpha-LE-1/2}, if
  \begin{align}  \label{e:Williamson-2}
    g(2^k) = o(b_k/u(2^k)), \quad k\toi,
  \end{align}
  then the SRT holds.  We next show that if $\tailr>0$, then
  \eqref{e:Williamson-2} is also a necessary condition for the SRT,
  which implies that the condition $\omega(x) = o(\rvf(x)^2)$ in
  \eqref{e:SRT-omega}  is not sufficient.  Denote $G(x,\delta)
  = \cbfrac{x}{\rvf(x)} \sum_{n < \rvf(\delta x)} \cdf\cvp n(x+J)$,
  where $J=(0,2h]$.  Since $\inf_{[0,1]}\limdf>0$, by Proposition
  \ref{p:sum-small-n-low}, there is $n_0\gg 1$, such that for
  $0<\delta \ll 1$, $x\gg 1$, and $T>0$,
  \begin{align*}
    G(x,\delta)
    &\gg
    \frac{1}{\rvf(x)^2} \sum_{n_0\le n<\rvf(\delta x)}
    n\int_0^1
    \omega(x-a_n t)\,\dd t
    \\
    &\gg
    \frac{g(x)^2}{x} \sum_{n_0\le n<\rvf(\delta x)}
    \frac{n}{a_n}\int_{x-a_n}^x
    [\omega(t)-T]^+\,\dd t.
  \end{align*}
  Set $T>0$ large enough.  Suppose there are
  infinitely many $k\in\Nats$, such that $\omega(2^k)\toi$; otherwise
  both the SRT and \eqref{e:Williamson-2} hold.  For such $x=2^k\gg 1$
  and $n_0\le n < \rvf(\delta x)$, by $x/2 < x-a_n < x-1$,
  \begin{align*}
    G(x,\delta)
    \gg
    \frac{g(x)^2}{x} \sum_{n_0\le n<\rvf(\delta x)} 
    \frac{n}{a_n} \Sbr{\frac{x}{b_k g(x)} - T}^+
    \asymp
    \frac{g(2^k)}{b_k} \sum_{n_0\le n<\rvf(\delta x)} n/a_n.
  \end{align*}
  From \eqref{e:an-inverse-rvf} and \eqref{e:rvf-diff},
  \begin{align*}
    \sum_{n_0\le n<\rvf(\delta x)} \lfrac{n}{a_n}
    \asymp
    \int_{n_0}^{\rvf(\delta x)}
    \frac{t\,\dd t}{\rvf^{-1}(t)}
    =
    \int_{\rvf^{-1}(n_0)}^{\delta x}
    \frac{\rvf(s)\,\dd \rvf(s)}{s}
    \asymp
    u(\delta x).
  \end{align*}
  Then by $u\in \RV_0$, $G(2^k,\delta) \gg \lfrac{g(2^k) u(2^k)}
  {b_k}$.  Then by \eqref{e:small-n-truncation}, it is seen that
  \eqref{e:Williamson-2} is a necessary condition for the SRT.
\end{example}

\section{Proofs of SRTs} \label{s:SRT-proofs}
To prove the SRTs in Section \ref{s:intro}, we only need a
relaxed version of Theorem \ref{t:sum-small-n}.
\begin{lemma} \label{l:R-bound}
  Let $L(x)\ge 1$ and $L(x)=o(\rvf(x))$ as $x\toi$.  Fix $\eta\in
  (0,1)$, $T\ge 0$, $r\in (0,1]$, and $1/2\le c_1<1\le c_2$.  Then
  given $\delta \in (0,1)$, for $x\gg 1/\delta$,
  \begin{align*}
    \XR(x, \delta)
    \ll 
    \frac{x}{\rvf(x)} 
    \sup_{t\ge c_1 x/\kappa} \frac{\iof_\eta(t,T)}{t \rvf(t)}
    \int_{\rvf^{-1}(L(x))}^{\delta x} 
    \cbfrac{\rvf(s)}{s}^2\,\dd s.
  \end{align*}
\end{lemma}
\begin{proof}
  It is clear that $\Iof_{\eta,n,r}(t,T) \le \iof_\eta(t,T)$
  for any $n\ge 0$.  For $c_1 x_n/\kappa \le t\le c_2 x_n + 2h$,
  $\tail\cdf(x_n)/x_n \asymp 1/(t \rvf(t))$.  Then
  \begin{align*}
    \frac{\tail\cdf(x_n)}{x_n}
    \sup_{c_1 x_n/\kappa \le t \le c_2 x_n+2h}
    \iof_\eta(t, T)
    \asymp
    \sup_{c_1 x_n/\kappa \le t \le c_2 x_n + 2h}
    \frac{\iof_\eta(t,T)}{t \rvf(t)}
    \le
    \sup_{t\ge c_1 x/\kappa}
    \frac{\iof_\eta(t,T)}{t \rvf(t)},
  \end{align*}
  so by \eqref{e:sum-small-n-R}, letting $\lambda(x) = \sum_{L(x)\le
    n<\rvf(\delta x)} \mean(N_n/a_{N_n})$,
  \begin{align*}
    \XR(x,\delta)
    \ll
    \frac{x}{\rvf(x)} 
    \sup_{t\ge c_1 x/\kappa} \frac{\iof_\eta(t,T)}{t \rvf(t)}
    \times  \lambda(x).
  \end{align*}
  It therefore only remains to show that for $x\gg 1/\delta$,
  \begin{align} \label{e:sum-n-an}
    \lambda(x)
    \ll
    \int_{\rvf^{-1}(L(x))}^{\delta x} 
    \cbfrac{\rvf(s)}{s}^2\,\dd s.
  \end{align}

  Write $w_n = N_n / a_{N_n}$.  Fix $p\in (0, \xpp)$ and let $u_n =
  w_n \cf{N_n\ge p n}$ and $v_n = w_n - u_n$.  By $x/\rvf^{-1}(x)\in 
  \RV$ with exponent $1-1/\alpha<0$, $w_n\ll 1$ and $u_n\ll (p
  n)/\rvf^{-1}(p n) \ll n/a_n$.  Then by $v_n = w_n \cf{N_n<p n} \ll
  \cf{N_n< p n}$, $\mean(v_n) \ll e^{-c n}$ for some $c>0$, giving
  $\mean(v_n) = o(n/a_n)$.  On the other hand, since $u_n/(n/a_n)\to
  (\xpp)^{1-1/\alpha}$ a.s., by dominated convergence, $\mean(u_n)
  \asymp n/a_n$.  Then by $\mean(u_n) \le \mean(w_n) = \mean(u_n) +
  \mean(v_n)$, $\mean(w_n) \asymp n/a_n$, which gives $\lambda(x)
  \ll\sum_{L(x)\le n<\rvf(\delta x)} \lfrac{n}{a_n}$.  Finally, by
  $L(x)/\rvf(x)\to 0$ and the same derivation at the end of Example
  \ref{ex:Williamson-2}, \eqref{e:sum-n-an} follows.
\end{proof}

\subsection{Proof of Theorem \ref{t:SRT-alpha-LE-1/2}}
From Remark \ref{r:small-n-truncation}, it suffices to prove the
following lemma.  Let $L$, $T$, and $\eta$ be as in Theorem
\ref{t:SRT-alpha-LE-1/2}.  Fix any $r\in (0,1]$ and $1/2\le c_1<1\le
c_2$.  For brevity, write $\Xr = \XR$.
\begin{lemma} \label{l:small-n-truncation}
  Let \eqref{e:diff2} and \eqref{e:diff} hold.  Then given any $r\in
  (0,1]$ and $1/2\le c_1<1\le c_2$,
  \begin{align} \label{e:RT-x-delta}
    \Lsup_{x\toi} \Xr(x, \delta) = 0 \quad
    \text{for any}\ \delta>0.
  \end{align}
  In particular, if \eqref{e:low-cut} also holds, then
  \eqref{e:small-n-truncation} holds.
\end{lemma}

\begin{proof}
  It suffices to show \eqref{e:RT-x-delta}, because then \eqref
  {e:small-n-truncation} follows from Theorem \ref{t:sum-small-n}.
  In the rest of the proof, let $\delta>0$ be fixed.  

  First suppose $\alpha \in (0,1/2)$.  Let $v(x) = L(x)^2/\rvf^{-1}
  (L(x))$.  Then by \eqref{e:diff2}, 
  \begin{align*}
    \sup_{t\ge c_1 x/\kappa} \frac{\iof_\eta(t,T)}{t \rvf(t)}
    =
    o\Grp{
      \sup_{t\ge c_1 x/\kappa}
      \frac{\rvf(t)}{t v(t)}
    }, \quad\text{as}\ x\toi.
  \end{align*}
  Since $L\in \RV_c$ with $0\le c\le \alpha$, 
  $\rvf^{-1}(L(x))\in \RV_{c/\alpha}$, so by definition of $v(t)$,
  $\lfrac{\rvf(t)} {[t v(t)]}\in \RV_b$ with $b = \alpha - 1 +
  c(1/\alpha-2) \le \alpha - 1 + \alpha(1/\alpha-2) = -\alpha < 0$.
  Consequently,
  \begin{align*}
    \sup_{t\ge c_1 x/\kappa}
    \frac{\rvf(t)}{t v(t)}
    \ll \frac{\rvf(x)}{x v(x)}.
  \end{align*}
  Combining the displays and Lemma \ref{l:R-bound}, as $x\toi$,
  \begin{align*}
    \Xr(x,\delta) = 
    o\Grp{
      \frac{1}{v(x)}
      \int_{\rvf^{-1}(L(x))}^\infty
      \cbfrac{\rvf(s)}{s}^2\,\dd s
    }.
  \end{align*}

  By assumption, $L(x)\ge 1$ for all $x>0$.  Since $\cbfrac{\rvf(x)}
  {x}^2 \in \RV$ with exponent $2\alpha-2<-1$, for $z\ge 1$,
  $\int_z^\infty \cbfrac{\rvf(s)}{s}^2\,\dd s \asymp \rvf(z)^2/z$.
  Letting $z=\rvf^{-1}(L(x))$, $\Xr(x, \delta) = o(1)$ as $x\toi$,
  hence proving \eqref{e:RT-x-delta}.

  Next suppose $\alpha = 1/2$.  By definition, $u(x) = \int_1^x
  \cbfrac{\rvf(s)}{s}^2\,\dd s$.  Since  $x/\rvf(x) \ll t/\rvf(t)$ for
  $x\ge 1$ and $t\ge c_1 x/\kappa$, by Lemma \ref{l:R-bound},
  \begin{align} \label{e:R-alpha-1/2-or-more}
    \Xr(x,\delta)
    \ll
    \sup_{t\ge c_1 x/\kappa} \frac{\iof_\eta(t,T)}{\rvf(t)^2}
    \times [u(\delta x) - u(\rvf^{-1}(L(x)))]
  \end{align}

  Since $\cbfrac{\rvf(x)}{x}^2\in \RV_{-1}$, $u\in \RV_0$.  By
  condition \eqref{e:diff}, if $\lfrac{u(x)}{u(\rvf^{-1}(L(x)))} \to
  1$, then $\lfrac{\iof_\eta(t,T)}{\rvf(t)^2} =O(1/u(t)) = O(1/u(x))$
  for $t\ge c_1 x/\kappa$; otherwise, $\lfrac{\iof_\eta(t,
    T)}{\rvf(t)^2} = o(1/u(x))$.  In either case, since $u$ is
  strictly increasing, $\Xr(x,\delta) = o(\lfrac{u(\delta x)}{u(x)}) =
  o(1)$ as $x\toi$, proving \eqref{e:RT-x-delta}.
\end{proof}

\subsection{Proof of Theorem \ref{t:SRT-alpha-GT-1/2}}
Similar to the argument for Theorem  \ref{t:SRT-alpha-LE-1/2}, it
suffices to show \eqref{e:small-n-truncation}.  Let $L(x)\equiv 1$,
$\eta=1/2$, and $T=0$.  Fix any $r\in (0,1]$, and $1/2\le c_1<1\le
c_2$.  Write $\Xr = \XR$.  By $\tail\cdf(y) \asymp
\tail\cdf(x)$ for $(1-\eta) x \le y \le x$,
\begin{align*}
  \iof_\eta(x,0)
  &=
  \int_{x/2}^x \frac{y\cdf(y+I)}{\tail\cdf(y)}\,\dd y
  \ll
  \frac{x}{\tail\cdf(x)}
  \int_{x/2}^x [\tail\cdf(y) - \tail\cdf(y+h)]\,\dd y.
\end{align*}
The last integral is equal to $\int_{x/2}^{x/2 + h} \tail\cdf -
\int_x^{x+h} \tail\cdf\le h \tail\cdf(x/2)$.  As a result,
\begin{align*}
  \iof_\eta(x,0)
  &\ll
  \frac{x}{\tail\cdf(x)} h\tail\cdf(x/2) = O(x).
\end{align*}
Then by Lemma \ref{l:R-bound} and $\cbfrac{\rvf(x)}{x}^2 \in \RV$ with
exponent $2\alpha - 2>-1$,
\begin{align*}
  \Xr(x,\delta)
  &\ll
  x\tail\cdf(x)^2 \int_1^{\delta x} \cbfrac{\rvf(y)}{y}^2\,\dd y
  \asymp
  x\tail\cdf(x)^2 \rvf(\delta x)^2/(\delta x)
  =
  O(\delta^{2\alpha-1}).
\end{align*}
Since $\alpha > 1/2$, by Theorem \ref{t:sum-small-n}, this then leads
to \eqref {e:small-n-truncation}.

\subsection{Proof of Theorem \ref{t:renewal-inf-div}}
Under \eqref{e:tail-ind}, $\tail\cdf(x) \sim \nu((x,\infty))$ and
$\cdf(-x) \sim \nu((-\infty, -x])$ as $x\toi$ \cite[][Th.~8.2.1]
{bingham:89}.  Then by \eqref{e:tail-ratio-ind}, $\cdf(-x) /
\tail\cdf(x)\to \tailr$.  Therefore, as the proof of Theorem \ref
{t:SRT-alpha-LE-1/2}, it suffices to show \eqref{e:small-n-truncation}
holds for $\cdf$.  Recall $X\sim \cum Y N  + W$, where $Y_1$, $Y_2$,
\ldots, $N$, and $W$ are independent, such that $Y_i \sim \cdf_\nu$,
$N \sim \dpois(\mu)$ with $\mu=\nu_1(\Reals)$, and $W$ is infinitely
divisible with $\levy$ measure $\nu(\cdot\cap (-r,r))$.  Henceforth,
by $\xi_\tau$ we mean $\cum Y \tau$ with $\tau$ a non-negative
integer valued random variable independent of all $Y_i$.  Then for
$n\ge 1$, $S_n \sim \xi_{N_n} + V_n$, where $N_n\sim \dpois(n\mu)$,
$V_n = \cum W n$, with $W_i\sim W$ being i.i.d.  Fix $M>1$, whose
value will be selected later.  Given $\rx\in (0,1/2)$, by
independence, for $\delta>0$ and $x>0$,
\begin{align*}
  \cdf\cvp n(x+I)
  &=
  \pr{\xi_{N_n} + V_n \in x + I}
  \\
  &\le
  \sum_{k\le \rvf(M\delta x)}
  \pr{\xi_k + V_n \in x +I\AND |V_n|\le \rx x}
  \,\pr{N_n = k}  + R_n(x)
  \\
  &\le
  \sum_{k\le \rvf(M\delta x)}
  \sup_{|t-x|\le \rx x}
  \cdf_\nu\cvp k(t+I)\,\pr{N_n = k} 
  + R_n(x)
\end{align*}
where $R_n(x) = \pr{N_n > \rvf(M\delta x)} + \pr{|V_n|>\rx x}$.
Summing over $n<\rvf(\delta x)$,
\begin{align} \label{e:small-n-inf-div}
  \sum_{n<\rvf(\delta x)} \cdf\cvp n(x+I)
  \le
  \sum_{k\le \rvf(M\delta x)}
  \sup_{|t-x|\le \rx x} \cdf\cvp k(t+I)
  \sumoi n \pr{N_n = k} 
  + \sum_{n<\rvf(\delta x)} R_n(x).
\end{align}
For each $k,n\ge 1$ and $s\in [n\mu, n\mu+\mu]$, $\pr{N_n = k} =
(n\mu)^k e^{-n\mu}/k!\le s^k e^{-s+\mu}/k!$.  Then
\begin{align*}
  \sumoi n \pr{N_n = k}
  \le \nth{k!} \intzi s^k e^{-s+\mu}\,\dd s
  = e^\mu.
\end{align*}
On the other hand, by Markov inequality,
\begin{align*}
  R_n(x)
  &\le
  \mean[e^{N_n - \rvf(M\delta x)}] + 
  \mean[e^{V_n - \rx x}+ e^{-V_n - \rx x}]
  \\
  &=
  \exp\{n\mu(e-1) - \rvf(M\delta x)\} +
  e^{-\rx x} [(\mean e^W)^n +(\mean e^{-W})^n].
\end{align*}
If $M>(4\mu)^{1/\alpha}$, then for all large $x>0$ and $n\le
\rvf(\delta x)$, $n\mu(e-1) \le 2\mu\rvf(\delta x) < \rvf(M\delta
x)/2$.   On the other hand, for any $t$, $\mean[e^{t W}]<\infty$
\cite[][Th.~25.17]{sato:99}.  If $c = \ln \mean[e^{|W|}]$, then
$(\mean e^{\pm W})^n \le e^{c n}$.  As a result,
\begin{align*}
  \sum_{n<\rvf(\delta x)} R_n(x)
  \le \rvf(\delta x) e^{- \rvf(M\delta x)/2} +
  \rvf(\delta x) e^{-\rx x + c\rvf(\delta x)}.
\end{align*}
Since $\rvf(x) = o(x)$, combining the above displays with
\eqref{e:small-n-inf-div} yields
\begin{align*}
  \Lsup_{x\toi}
  \frac{x}{\rvf(x)} \sum_{n<\rvf(\delta x)}
  \cdf\cvp n(x+I)
  \le
  \Lsup_{x\toi} \frac{x}{\rvf(x)}
  \sum_{k\le \rvf(M\delta x)}
  \sup_{|t-x|\le \rx x} \cdf_\nu\cvp k(t+I).
\end{align*}
Since by assumption \eqref{e:diff2} and \eqref{e:diff} hold for
$\cdf_\nu$, by Theorem \ref{t:sum-small-n} and Lemma \ref
{l:small-n-truncation}, the left hand side (LHS) is dominated by
\begin{align*}
  o(\delta^\alpha) + O(\delta^{2\alpha}) T
  + \Lsup_{x\toi} \frac{x}{\rvf(x)}
  \sum_{k<L(x)}
  \sup_{|t-x|\le \rx x} \cdf_\nu\cvp k(t+I).
\end{align*}
Since $\rx>0$ is arbitrary, by \eqref{e:low-cut-inf-div}, the LHS of
the previous display is dominated by $o(\delta^\alpha) +
O(\delta^{2\alpha}) T$.  Thus \eqref{e:small-n-truncation} holds for
$\cdf$, finishing the proof.

\subsection{Proof of Theorem \ref{t:ladder}}
Let $\omega_+(x)$ and $\iof_{+,\eta}(x,T)$ denote the functions
defined by \eqref{e:prob-ratio} and \eqref{e:overflow} respectively
in terms of $\cdf\!_+$.  Recall the well-known identities
\cite[][p.~399]{feller:71}
\begin{align}\label{e:renewal-ladder}
  \begin{split}
    \cdf\!_+(\dd t)
    &=
    \intzi \cdf(y+\dd t)\, \renf_-(\dd y), \quad t>0,
    \\
    \cdf\!_-(\dd t)
    &=
    \int_{0+}^\infty \cdf(-y-\dd t)\, \renf_+(\dd y), \quad t\ge 0.
  \end{split}
\end{align}
For $t>0$, by the first identity in \eqref{e:renewal-ladder},
\begin{align*}
  [\omega_+(t)-T]^+
  &=
  \frac{[t \cdf\!_+(t+I) - \tail\cdf\!_+(t) T]^+}{\tail\cdf\!_+(t)}
  \\
  &\le
  \nth{\tail\cdf\!_+(t)} \intzi
  [t \cdf(t+y+I) - \tail\cdf(t+y) T]^+\,\renf_-(\dd y)
  \\
  &=
  \frac{t}{\tail\cdf\!_+(t)} \intzi
  \frac{\tail\cdf(t+y)}{t+y}
  \Sbr{
    \omega(t+y) - \frac{(t+y) T}{t}
  }^+\,\renf_-(\dd y).
\end{align*}
Given $x\ge 1$, integrate both sides of the inequality over $t\in
[(1-\eta) x, x]$.  For each such $t$ and $y\ge 0$,
$t/\tail\cdf\!_+(t)\asymp x/\tail\cdf\!_+(x)$ and
$\tail\cdf(t+y)/(t+y)\asymp \tail\cdf(x+y)/(x+y)$.  Then
\begin{align*}
  \iof_{+,\eta}(x, T)
  &\le
  \int_{(1-\eta)x}^x \frac{t\,\dd t}{\tail\cdf\!_+(t)}
  \intzi 
  \frac{\tail\cdf(t+y)}{t+y} [\omega(t+y)-T]^+
  \,\renf_-(\dd y)
  \\
  &\ll
  \frac{x}{\tail\cdf\!_+(x)}
  \intzi
  \frac{\tail\cdf(x+y)\renf_-(\dd y)}{x+y} 
  \int_{(1-\eta)x}^x
  [\omega(t+y) - T]^+\,\dd t.
\end{align*}
Since
\begin{align*}
  \int_{(1-\eta)x}^x[\omega(t+y) - T]^+\,\dd t
  &=
  \int_{(1-\eta)x + y}^{x+y} [\omega(t) - T]^+\,\dd t
  \\
  &\le
  \int_{(1-\eta)(x + y)}^{x+y} [\omega(t) - T]^+\,\dd t
  =
  \iof_\eta(x+y,T),
\end{align*}
then
\begin{align*}
  \iof_{+,\eta}(x,T)
  \ll
  \frac{x}{\tail\cdf_+(x)}
  \intzi \frac{\iof_\eta(x+y,T)\,\renf_-(\dd y)}
  {(x+y)\rvf(x+y)}.
\end{align*}

First, suppose $\alpha\ppar\in (0, 1/2)$.   Let
\begin{align*}
  Q(x) = \frac{x L(x)^2 \rvf(x)}{\rvf_+(x)^2 \rvf_+^-(L(x))}.
\end{align*}
Then by \eqref{e:diff2-ladder},
\begin{align*}
  \iof_{+,\eta}(x,T) = o\Grp{
    x\rvf_+(x) \intzi \frac{\renf_-(\dd y)}{Q(x+y)}
  }, \quad
  x\toi.
\end{align*}
By \eqref{e:low-cut-ladder}, $L\in \RV_\theta$ with $\theta \in [0,
\alpha\ppar]$.  Then $Q\in\RV_\beta$ with $\beta=1 + 2\theta + \alpha 
-2\alpha\ppar - \theta/(\alpha\ppar) \ge \alpha$.  Without
loss of generality, assume that $Q$ is smooth and strictly increasing.
Then
\begin{align*}
  \intzi \frac{\renf_-(\dd y)}{Q(x+y)}
  &=
  \intzi \renf_-(\dd y) \int_y^\infty \frac{Q'(x+t)\,\dd t}{Q(x+t)^2}
  \\
  &\asymp
  \intzi \frac{\renf_-(t)\,\dd t}{(x+t) Q(x+t)},
\end{align*}
where the second line is by $Q'(x)\sim \beta Q(x)/x$ as $x\toi$ and
Fubini's Theorem.   To continue, we need the following result.
\begin{lemma} \label{l:ladder-renewal-measure}
  If $\ppar\in (0,1]$, then $\renf_-(x) \ll \rvf(x)/\rvf_+(x)$ as
  $x\toi$.
\end{lemma}
Assuming the lemma to be true for now, by change of variable $t= x s$,
\begin{align*}
  \intzi \frac{\renf_-(t)\,\dd t}{(x+t) Q(x+t)}
  &\ll
  \intzi \frac{[\rvf(x s)/\rvf_+(x s)]\,\dd s}{(1+s) Q(x(1+s))}
  \\
  &\sim
  \frac{\rvf(x)}{\rvf_+(x) Q(x)}
  \intzi \frac{s^{\alpha(1-\ppar)}\,\dd s}{(1+s)^{1+\beta}}.
\end{align*}
Consequently, $\iof_{+,\eta}(x,T) = o\Grp{\lfrac{\rvf_+(x)^2
    \rvf_+^-(L(x))}{L(x)^2}}$, so by Theorem \ref{t:SRT-alpha-LE-1/2},
the SRT holds for $\cdf\!_+$.   The case $\alpha\ppar=1/2$ can be
proved similarly by letting $Q(x) = x \rvf(x)/k_+(x)$.

\begin{proof}[Proof of Lemma \ref{l:ladder-renewal-measure}]
  More exact asymptotic result than the lemma probably is already
  known.  We prove the lemma only for completeness.  From
  \eqref{e:renewal-ladder},
  \begin{align*}
    \cdf\!_+(x) \asymp
    \intzi \frac{\renf_-(\dd y)}{\rvf(x+y)}
    \asymp 
    \intzi \frac{\renf_-(t)\,\dd t}{(x+t) \rvf(x+t)}.
  \end{align*}
  If $\ppar\in (0,1)$, then $\renf_-\in \RV_{\alpha(1-\ppar)}$.  Then
  by change of variable $t = x s$, RHS $\asymp \renf_-(x)/\rvf(x)$,
  yielding $\renf_-(x) \asymp \rvf(x)/\rvf_+(x)$.  If $\ppar = 1$,
  then
  \begin{align*}
    \text{RHS}
    &\ge
    \renf_-(x)\int_x^\infty \frac{\dd t}{(x+t)\rvf(x+t)}
    \asymp
    \renf_-(x)/\rvf(x),
  \end{align*}
  finishing the proof.
\end{proof}

\section{Proofs of technical tools} \label{s:small-n}
Recall that $\kappa = \Flr{1/\alpha}$.  In this section, let
$\gamma\in (\alpha^{-1} (\kappa+1)^{-1}, 1)$ be fixed and as in
\cite{doney:97}, define $\zeta_{n,x} = a_n^{1-\gamma} x^\gamma$.

\subsection{Lemmas for Theorem \ref{t:sum-small-n}}
Given $n\ge 1$, for $k\le
n$, $X_k$ will be colloquially referred to as a ``large jump'' in
$S_n$ if $X_k > \zeta_{n,x}$, and as a ``small step'' otherwise.
Given $\delta>0$, $\eno X n$ is called ``short'' if $n < \rvf(\delta
x)$.

For $n\ge 1$, denote by $X_{n:1} \ge X_{n:2} \ge \ldots \ge X_{n:n}$
the arrangement of $\eno X n$ in decreasing order and $S_{n:k} =
X_{n:1} + \cdots + X_{n:k}$.  For $n\ge 0$ and $k>n$, define
$X_{n:0}=\infty$ and $X_{n:k}=0$.  Then for $n\ge 1$ and $0\le k\le
n$, $X_{n:k}>\zeta_{n,x}\Iff$ there are at least $k$ large jumps in
$S_n$, and $X_{n:k}>\zeta_{n,x}\ge X_{n:k+1}\Iff$ there are exactly
$k$ large jumps in $S_n$.

In the following, Lemmas \ref{l:many-large-jumps} and
\ref{l:small-sum} show that among all short sequences, one only need
consider those with at most $\kappa$ large jumps and with sum
of large jumps at least $(1-\rx) x$, where $\rx>0$ is arbitrary.
Under the extra condition that $X>0$, for each short sequence with
the above property, Lemma \ref{l:large-sum} bounds its contribution by
separately considering its jumps that are smaller than $r a_n$ and
those that are greater than $r a_n$.

\begin{lemma} \label{l:many-large-jumps}
  Let $k=\kappa+1$ and $b = [1/(k\gamma) + \alpha]/2$.  Note
  that $k\gamma b>1$.
  \begin{enumerate}[topsep=1ex, itemsep=0ex, leftmargin=4ex,
    label={\arabic*)}]
  \item Given $0<p\le 1$, for $\delta\in (0,1)$,
    \begin{align*} 
      \Lsup_{x\toi}
      \frac{x}{\rvf(x)} \sum_{n<\rvf(\delta x)}
      \sup_{p n\le m\le n\AND z\ge x}
      \pr{
        S_m \in z + I\AND X_{m:k}> \zeta_{m,z}
      }
      \ll_p\delta^{k\gamma b+\alpha-1}.
    \end{align*}
    
  \item For any $\ell\in\RV_0$,
    \begin{align*} 
      \Lsup_{x\toi}
      \frac{x}{\rvf(x)} \sum_{n\le \ell(x)}
      \sup_{m\le n\AND z\ge x}
      \pr{
        X_{m:k}> \zeta_{m,z}
      }
      = 0.
    \end{align*}
  \end{enumerate}
\end{lemma}

\begin{remark}
  If $X>0$, then to prove the SRT, instead of 1) and 2), it
  suffices to establish
  \begin{align*} 
    \Lsup_{x\toi}
    \frac{x}{\rvf(x)} \sum_{n<\rvf(\delta x)}
    \pr{
      S_n \in x + I\AND X_{n:k}> \zeta_{n,x}
    }
    \ll  \delta^{k\gamma b+\alpha-1}.
  \end{align*}
\end{remark}

\begin{proof}
  1) Given $z\ge x$, for $m<k$, $\pr{S_m\in z+I\AND X_{m:k} >
    \zeta_{m,z}} = 0$, and for $m\ge k$, $\pr{S_m\in z+I\AND X_{m:k}
    > \zeta_{m,z}}\le m^k \Delta_{m,z}$, where $\Delta_{m,z} = \pr{S_m
    \in z+I\AND X_{k:k} > \zeta_{m,z}}$.  Therefore it suffices to
  show
  \begin{align} \label{e:finite-k}
    \Lsup_{x\toi}
    \frac{x}{\rvf(x)} \sum_{k\le n<\rvf(\delta x)}
    n^k \sup_{p n\vee k\le m\le n\AND z\ge x}
    \Delta_{m,z}
    \ll_p \delta^{k\gamma b+\alpha-1}.
  \end{align}
  For $z\ge x$ and $(p n)\vee k\le m\le n$, by independence of the
  $X_i$'s,
  \begin{align*}
    \Delta_{m,z}
    &=
    \int_{\text{all}\ y_i > \zeta_{m,z}}
    \pr{
      S_{m-k} \in z - y_1 - \cdots - y_k + I
    }\,
    \cdf(\dd y_1) \cdots \cdf(\dd y_k).
  \end{align*}
  From the LLTs and $\sup\limdf<\infty$, $\sup_s \pr{S_{m-k} \in s+I}
  \ll 1/a_m \ll_p 1/a_n$ \cite[][Th.~8.4.1--2]{bingham:89}, so by
  the display $\Delta_{m,z} \ll_{p,k}
  \tail\cdf(\zeta_{m,z})^k/a_n$.  Since $\zeta_{m,z} =  a_m^{1-\gamma}
  z^\gamma \asymp_p a_n^{1-\gamma} z^\gamma \ge \zeta_{n,x}$, then
  \begin{align} \label{e:finite-k2}
    \sup_{p n\vee k\le m\le n\AND z\ge x} \Delta_{m,z}
    \ll_{p,k} \tail\cdf(\zeta_{n,x})^k/a_n.
  \end{align}
  By \eqref{e:an-inverse-rvf} and change of variable $s =
  \rvf^{-1}(t)$,
  \begin{align*}
    \sum_{k\le n<\rvf(\delta x)} 
    n^k \tail \cdf(\zeta_{n,x})^k/a_n
    &\asymp
    \int_k^{\rvf(\delta x)}
    \frac{
      t^k \tail \cdf(\rvf^{-1}(t)^{1-\gamma} x^\gamma)^k
    }
    {\rvf^{-1}(t)}\,\dd t
    \\
    &=
    \int_{\rvf^{-1}(k)}^{\delta x}
    \frac{\rvf(s)^k \tail \cdf(s^{1-\gamma} x^\gamma)^k}
    {s} \rvf'(s)\,\dd s.
  \end{align*}
  Then by \eqref{e:rvf-diff}, for $x\ge \rvf^{-1}(1+k)/\delta$, 
  \begin{align*}
    \sum_{k\le n<\rvf(\delta x)} 
    n^k \tail \cdf(\zeta_{n,x})^k/a_n
    &\ll
    \int_{\rvf^{-1}(k)}^{\delta x}
    \frac{\tail \cdf(s^{1-\gamma} x^\gamma)^k}
    {s} \frac{\rvf(s)^{k+1}}{s} \,\dd s
    \\
    &\ll
    \delta^\alpha
    \rvf(x)
    \int_{\rvf^{-1}(k)}^{\delta x}
    \frac{\tail \cdf(s^{1-\gamma} x^\gamma)^k \rvf(s)^k}
    {s^2} \,\dd s.
    \\
    &\ll
    \delta^\alpha
    \rvf(x)
    \int_{\rvf^{-1}(k)}^{\delta x} \nth{s^2}
    \Sbr{
      \frac{\rvf(s)}{\rvf(s^{1-\gamma} x^\gamma)}
    }^k \,\dd s.
  \end{align*}
  For $x\gg 1/\delta$ and $\rvf^{-1}(k)\le s\le \delta x$, by Potter's
  Theorem  \cite[] [Th.~1.5.6]{bingham:89}, $\rvf(s) /
  \rvf(s^{1-\gamma} x^\gamma) \ll_b (s/x)^{b\gamma}$.  Then by $k\gamma
  b>1$, 
  \begin{align*}
    \int_{\rvf^{-1}(k)}^{\delta x} \nth{s^2}
    \Sbr{
      \frac{\rvf(s)}{\rvf(s^{1-\gamma} x^\gamma)}
    }^k \,\dd s
    &\ll_b
    \int_{\rvf^{-1}(k)}^{\delta x} 
    \frac{(s/x)^{k\gamma b}}{s^2}\,\dd s
    \ll
    \nth x\frac{\delta^{k\gamma b-1}}{k\gamma b-1}.
  \end{align*}
  Combining the above two displays,
  \begin{align*}
    \sum_{k\le n<\rvf(\delta x)} 
    n^k \tail \cdf(\zeta_{n,x})^k/a_n
    \ll_b \frac{\rvf(x)}{x}
    \frac{\delta^{k\gamma b+\alpha-1}}{k\gamma b-1},
  \end{align*}
  which together with \eqref{e:finite-k2} leads to
  \eqref{e:finite-k}.

  2) By $\zeta_{m,z}\ge a_1^{1-\gamma} z^\gamma = z^\gamma$,
  $\pr{X_{m:k}>\zeta_{m,z}} \le m^k \tail\cdf(z^\gamma)^k \le n^k
  \tail\cdf(x^\gamma)^k$ for $z\ge x$ and $m\le n$.  Then from
  \begin{align*}
    \frac{x}{\rvf(x)} \sum_{n\le \ell(x)} n^k \tail\cdf(x^\gamma)^k
    \ll_k
    \frac{x \ell(x)^{k+1}}{\rvf(x) \rvf(x^\gamma)^k} \in
    \RV_{1-\alpha- \alpha\gamma k},
  \end{align*}
  the proof is complete.
\end{proof}

For Lemmas \ref{l:small-sum} and \ref{l:large-sum}, the following
local large deviation bound will be used.
\begin{lemma} \label{l:local-ldp}
  Let \eqref{e:tail-RV} hold and $\alpha \in (0,1)$.  Given $s_0>0$,
  there is a constant $c\ge 0$ only depending on $\{\cdf, \rvf,
  s_0\}$, such that for all $x>0$, $s\ge s_0$, and $n\ge 1$,
  \begin{align*}
    \pr{S_n \in x+I\AND X_{n:1}\le s}
    \ll (1/s + 1/a_n) e^{-x/s+ c n/\rvf(s)}.
  \end{align*}
\end{lemma}

\begin{remark}
  The bound is uniform in all $n\ge 1$ and $s\ge s_0$, and it only
  needs an assumption on the right-tail of $\cdf$.  The implicit
  constant in ``$\ll$'' is independent of $s_0$.  The proof of the
  bound follows \cite{denisov:08:ap}; also see \cite{doney:97,
    gouezel:11} for results restricted to the arithmetic or operator
  cases.
\end{remark}

\begin{proof}
  For $p\ge 0$ and $s>0$, define $\mu_p(s) = \mean[X^p\cf{0<X\le s}]$.
  By integration by parts and Karamata's Theorem \cite[][Th.~1.5.11]
  {bingham:89}, for $p\ge 1$,
  \begin{align}
    \mu_p(s)
    &=
    p \int_0^s \tail\cdf(u) u^{p-1}\,\dd u - \tail\cdf(s) s^p
    \sim
    \frac{\alpha s^p}{(p-\alpha) \rvf(s)}\toi, \quad s\toi.
    \label{e:Karamata-p}
  \end{align}
  As a result, for $s\ge 1$, $\mu_2(s)/\mu_1(s)^2 \asymp \rvf(s)$.
  Since $\mu_0(s)\to \xpp>0$ as $s\toi$, one can fix $\theta\ge 1$
  which only depends on $\{\cdf, \rvf\}$, such that
  \begin{align} \label{e:cut-moment}
    \mu_2(s)\mu_0(s) > 2\mu_1(s)^2, \quad s\ge \theta.
  \end{align}

  We now follow the proof for Lemma~7.1(iv) and Proposition~7.1 in
  \cite{denisov:08:ap}.  The case $\cdf(s)=0$ is trivial.  Let
  $\cdf(s)>0$ and $\psi(s) = \mean[e^{X/s} \cf{X\le s}]$.  Then
  $0<\psi(s)\le e$ and $G_s(\dd x) = \psi(s)^{-1}   e^{x/s} \pr{X\in
    \dd x\AND X\le s}$ is a probability measure.  By $\ln\mean Z =
  \ln[1+\mean(Z-1)] \le \mean(Z-1)$ for $Z\ge 0$
  \begin{align*}
    \ln \psi(s)
    \le
    \mean[e^{X/s}\cf{X\le s}-1]
    &\le
    \mean[(e^{X/s}-1)\cf{X\le s}]
    \\
    &\le
    \mean[(e^{X/s}-1)\cf{0<X\le s}].
  \end{align*}
  Since $e^x-1\le 2 x$ for $x\in [0,1]$, $\ln \psi(s) \le 2 s^{-1}
  \mean[X\cf{0<X\le s}] = 2 \mu_1(s)/s$.  Then by
  \eqref{e:Karamata-p}, there is $c\ge 0$ only 
  depending on $\{\cdf,\rvf, s_0\}$, such that $\ln\psi(s) \le
  c/\rvf(s)$ for $s\ge s_0$.  Let $\tilde S_n = \cum{\tilde X} n$ with
  $\tilde X_i$ i.i.d.\ $\sim G_s$.  Then
  \begin{align*}
    \pr{S_n \in x+I\AND X_{n:1}\le s}
    &=
    \psi(s)^n
    \mean[e^{-\tilde S_n/s}\cfx{\tilde S_n\in x+I}]
    \\
    &\le
    \psi(s)^n e^{-x/s} \pr{\tilde S_n \in x + I}
    \\
    &\le
    e^{-x/s + c n/\rvf(s)} \pr{\tilde S_n \in x + I}.
  \end{align*}
  Denote $f_s(t) = \mean[e^{\iunit t\tilde X}]$ with $\tilde X \sim
  G_s$.  By \cite{petrov:95}, Lemma 1.16, 
  \begin{align*}
    \pr{\tilde S_n \in x + I}
    &\le
    (96/95)^2
    (h\vee \theta) \int_{-1/\theta}^{1/\theta}|f_s(t)|^n \,\dd t
    \\
    &\ll
    1/s + \int_{1/(s\vee \theta)}^{1/\theta} |f_s(t)|^n \,\dd t.
  \end{align*}
  From the displays, it is seen that to finish the proof, it suffices
  to show that for $s\ge \theta$,
  \begin{align} \label{e:tilted-S}
    \int_{1/s}^{1/\theta} |f_s(t)|^n \,\dd t \ll 1/a_n.
  \end{align}
  
  Let $\xi = \tilde X_1 - \tilde X_2$.  Then $|f_s(t)|^2
  = \mean[e^{\iunit t\xi}] = \mean\cos t\xi>0$.  Since $x\le
  \exp\{-(1-x^2)/2\}$,
  \begin{align} \label{e:abs-tilde-X}
    |f_s(t)|
    =
    \sqrt{\mean\cos t\xi}
    \le
    \exp\Cbr{-(1-\mean\cos t\xi)/2}.
  \end{align}
  Since $1-\cos z\ge C z^2$ for $|z|\le 1$, where $C>0$ is an
  absolute constant, for $t\in [1/s,1/\theta]$,
  \begin{align*}
    1-\mean\cos t\xi
    &\ge
    \mean[(1-\cos t\xi)\cf{|\xi|\le 1/t}]
    \\
    &\ge
    C t^2 \mean[(\tilde X_1 - \tilde X_2)^2 \cfx{0< \tilde X_1,
      \tilde X_2\le 1/t}]
    \\
    &=
    C \psi(s)^{-2} t^2 \mean[(X_1 - X_2)^2 e^{(X_1+X_2)/s}
    \cf{0<X_1, X_2\le 1/t}]
    \\
    &\ge
    (C/e^2) t^2 \mean[(X_1 - X_2)^2 \cf{0<X_1, X_2\le 1/t}].
  \end{align*}
  Since the expectation on the last line is $2\mu_2(1/t) \mu_0(1/t) -
  2\mu_1(1/t)^2$, by \eqref{e:Karamata-p} and \eqref{e:cut-moment}, 
  \begin{align*}
    1-\mean\cos t\xi
    \gg
    t^2 \mu_2(1/t) \mu_0(1/t) \gg 1/\rvf(1/t).
  \end{align*}
  Combining this with \eqref{e:abs-tilde-X}, for some constant
  $b>0$,
  \begin{align*}
    \int_{1/s}^{1/\theta}  |f_s(t)|^n \,\dd t
    \le
    \int_{1/s}^{1/\theta} \exp\Cbr{-\frac{b n}{\rvf(1/t)}}\,\dd t.
  \end{align*}
  By Potter's Theorem \cite[][Th.~1.5.6]{bingham:89},
  \begin{align*}
    \frac{n}{\rvf(1/t)}
    =\frac{\rvf(a_n)}{\rvf(1/t)}
    \gg
    \min\{(a_n t)^{\alpha/2}, (a_n t)^{3\alpha/2}\}.
  \end{align*}
  On the other hand, for any $c>0$ and $q>0$,
  \begin{align*}
    \int_{1/s}^{1/\theta} e^{-c (a_n t)^q}\,\dd t
    = \nth{a_n} \int_{a_n/s}^{a_n/\theta} e^{-c t^q}\,\dd t
    \le \nth{a_n} \intzi e^{-c t^q}\,\dd t < \infty.
  \end{align*}
  Then \eqref{e:tilted-S} easily follows.
\end{proof}

\comment{The next step is to use Lemma \ref{l:local-ldp} to show that among
short sequences with a given number of large jumps, those with a small
sum of large jumps can be ignored.}  For $k\ge 0$, $n\ge 1$ and $x>0$,
denote
\begin{align}  \label{e:events}
  \begin{split}
    E_{n,k,x}
    &=
    \{S_n \in x+I\AND X_{n:k}>\zeta_{n,x}\ge X_{n:k+1}\},
    \\
    \Gamma_{n,k,x}
    &=
    \{S_n\in x + I\AND X_i>\zeta_{n,x}\ge X_j \text{ for all }
    i\le k < j\},
  \end{split}
\end{align}
where both sets are defined to be $\emptyset$ if $k>n$.

\begin{lemma} \label{l:small-sum}
  Fix $k\ge 0$, $\rx\in (0,1)$, and $M>1$.  Then for $0<\delta
  \ll_{\rx,M} 1$,
  \begin{align*}
    \Lsup_{x\toi}
    \frac{x}{\rvf(x)} \sum_{n<\rvf(\delta x)}
    \sup_{z\ge x\AND m\le n}
    \pr{
      E_{m,k,z}\AND
      S_{m:k}\le (1-\rx)z
    } \ll_k \delta^M.
  \end{align*}
\end{lemma}

\begin{remark}
  If $X>0$, then to prove the SRT for $X>0$, it suffices to establish
  \begin{align*}
    \Lsup_{x\toi}
    \frac{x}{\rvf(x)} \sum_{n<\rvf(\delta x)}
    \pr{
      E_{n,k,x}\AND
      S_{n:k}\le (1-\rx)x
    } \ll_k \delta^M.
  \end{align*}
\end{remark}

\begin{proof}
  Let $\Delta_{n,x} = \pr{\Gamma_{n,k,x}\AND
    S_k\le (1-\rx)x}$.   Since $\pr{E_{n,k,x}\AND S_{n:k}\le (1-\rx)
    x} \le n^k \Delta_{n,x}$, it suffices to show for $0< \delta
  \ll_{\rx,M} 1$,
  \begin{align} \label{e:small-sum-0}
    \Lsup_{x\toi}
    \frac{x}{\rvf(x)} \sum_{n<\rvf(\delta x)}
    \sup_{z\ge x\AND m\le n}
    m^k \Delta_{m,z} \ll_k \delta^M.
  \end{align}
  By definition of $\Gamma_{m,k,z}$, it suffices to consider $m\ge
  k$.    For $j\ge 1$, let $X'_j = X_{k+j}$ and $S'_j =
  \cum{X'} j$.   If $k\ge 1$, then for $z\ge 1$,
  \begin{align}
    \Delta_{m,z}
    &=
    \pr{S_k + S'_{m-k}\in z + I\AND X_{k:k} > \zeta_{m,z} \ge
      X'_{m-k:1}\AND S_k \le (1-\rx)z
    }
    \nonumber \\
    &\le
    \pr{
      S'_{m-k} \in (z-S_k+I)\cap [\rx z,\infty)\AND
      X_{k:k}>\zeta_{m,z} \ge X'_{m-k:1}
    }
    \nonumber \\
    &\le
    \sup_{t\ge \rx z} \pr{
      S'_{m-k} \in t+I\AND
      X_{k:k}>\zeta_{m,z} \ge X'_{m-k:1}
    }
    =
    \tail\cdf(\zeta_{m,z})^k M_{m,z},
    \label{e:small-sum-1}
  \end{align}
  where $M_{m,z} = \sup_{t\ge \rx z} \pr{S_{m-k} \in t+I\AND
    X_{m-k:1}\le \zeta_{m,z}}$.   As $\zeta_{m,z}\ge 1$,
  by Lemma \ref{l:local-ldp}, there is a constant $c$ only depending
  on  $\{\cdf, \rvf\}$, such that 
  \begin{align} \label{e:small-sum-2}
    M_{m,z}
    \ll_k
    (1/\zeta_{m,z} + 1/a_m) e^{-\rx z/\zeta_{m,z} + 
      c m /\rvf(\zeta_{m,z})}.
  \end{align}
  It is easy to see \eqref{e:small-sum-2} still holds if $k=0$ and
  $m\ge 1$.  For $\delta\in (0,1)$, $z\ge x\ge 1$ and $m\le n<
  \rvf(\delta x)$,
  \begin{align*}
    a_m/\zeta_{m,z}
    &=
    (a_m/z)^\gamma \le (a_n / x)^\gamma \le \delta^\gamma,
    \\
    m/\rvf(\zeta_{m,z})
    &=
    \rvf(a_m)/\rvf(\zeta_{m,z}) \ll \delta^{\alpha\gamma/2},
    \\
    z/\zeta_{m,z}
    &=
    (z/a_m)^{1 - \gamma} \ge (x/a_m)^{1-\gamma} \ge
    (x/a_n)^{1-\gamma} \ge \delta^{\gamma-1}.
  \end{align*}
  Combining the bounds with \eqref{e:small-sum-2}, for all $\delta>0$
  small enough and $z\ge 1$ large enough,
  \begin{align*}
    M_{m,z}
    \ll_k
    (1/a_m) e^{-(\rx/2)(x/a_m)^{1-\gamma}}
    =
    x^{-1} h((x/a_m)^{1-\gamma}),
  \end{align*}
  where $h(u) =u^{1/(1-\gamma)}e^{-\rx u/2}$.  As already seen,
  $(x/a_m)^{1-\gamma} \ge (x/a_n)^{1-\gamma} \ge 1/\delta^{1-\gamma}$.
  Since $h(u)$ is decreasing for $u$ large enough, for
  $0<\delta\ll_{\rx,\gamma} 1$,
  \begin{align*}
    M_{m,z}
    \ll_k
    x^{-1}
    h((x/a_n)^{1-\gamma}) = (1/a_n) e^{-(\rx/2)(x/a_n)^{1-\gamma}}.
  \end{align*}
  Then by \eqref{e:small-sum-1} and $m\tail\cdf(\zeta_{m,z}) \ll
  \rvf(a_m)/\rvf(\zeta_{m,z}) \ll 1$,
  \begin{align*}
    m^k \Delta_{m,z}
    \ll_k
    \frac{m^k \tail\cdf(\zeta_{m,z})^k}{a_n}
    e^{-(\rx/2)(x/a_n)^{1-\gamma}}
    \le
    \nth{a_n}
    e^{-(\rx/2)(x/a_n)^{1-\gamma}}.
  \end{align*}
  As a result, given $\delta>0$ small enough, for large $x>0$,
  \begin{align*}
    \sum_{n<\rvf(\delta x)}
    \sup_{z\ge x\AND m\le n}
    m^k \Delta_{m,z}
    &\ll_k
    \sum_{n<\rvf(\delta x)}
    \nth{a_n} e^{-(\rx/2)(x/a_n)^{1-\gamma}}
    \\
    &\ll
    \int_1^{\rvf(2\delta x)}
    \nth{\rvf^{-1}(t)}
    e^{-(\rx/2)(x/\rvf^{-1}(t))^{1-\gamma}}
    \,\dd t,
  \end{align*}
  where the second line follows from \eqref{e:an-inverse-rvf}.  By
  change of variable $s=x/\rvf^{-1}(t)$ and \eqref{e:rvf-diff}, the
  last integral is no greater than
  \begin{align*}
    \nth x
    \int_{1/2\delta}^x s e^{-(\rx/2)s^{1-\gamma}}\,
    |\dd \rvf(x/s)|
    &=
    \nth x
    \int_{1/2\delta}^x
    (x/s)\rvf'(x/s) e^{-(\rx/2)s^{1-\gamma}}
    \,\dd s
    \\
    &\ll
    \nth{x}
    \int_{1/2\delta}^x
    \rvf(x/s) e^{-(\rx/2) s^{1-\gamma}}\,\dd s
    \\
    &\ll
    \frac{\delta^\alpha \rvf(x)}{x}
    \int_{1/2\delta}^\infty e^{-(\rx/2) s^{1-\gamma}} \,\dd s.
  \end{align*}
  Therefore,
  \begin{align*}
    \frac{x}{\rvf(x)} \sum_{n<\rvf(\delta x)}
    \sup_{z\ge x\AND m\le n} m^k \Delta_{m,z}
    \ll_k
    \delta^\alpha \int_{1/2\delta}^\infty e^{-(\rx/2) s^{1-\gamma}}
    \,\dd s,
  \end{align*}
  which yields \eqref{e:small-sum-0}, as the RHS is
  $O(\delta^M)$ for $0<\delta \ll_{\rx,\gamma,M} 1$.
\end{proof}

The last lemma in this subsection requires $X>0$.  Let $r\in (0,1]$ be
fixed.  Define 
\begin{align*}
  \tau_n = \sum_{i=1}^n \cf{X_i>r a_n},\quad
  S'_n = \sum_{i=1}^n X_i\cf{X_i> r a_n}.
\end{align*}
Since $n\tail\cdf(r a_n) \to r^{-\alpha}$ as $n\toi$ and $\tau_n \sim
\dbin(n, \tail\cdf(r a_n))$, letting $\theta = 2 r^{-\alpha}$, for all
$n\ge 1$ and $m\ge 0$, $\pr{\tau_n = m} \ll_r \theta^m/m!$.
Conditioning on $\tau_n=k$, $S_n-S'_n$ and $S'_n$ are independent such
that $S_n - S'_n \sim B_{n-k} = \sum_{i=1}^{n-k} b_i$ and $S'_n \sim
U_k = \sum_{i=1}^k u_i$, with $b_i$ i.i.d.\ $\sim \cf{x\le r
  a_n}\cdf(\dd x)/\cdf(r a_n)$ and $u_i$ i.i.d.\ $\sim \cf{x> r
  a_n} \cdf(\dd x) /\tail\cdf(r a_n)$.  In the definition of $b_i$, if
$\cdf(r a_n)=0$, then let $b_i\equiv 0$.

\begin{lemma} \label{l:large-sum}
  Let $X>0$.  Fix $k\ge 1$ and $\eta\in (0,1)$.  Then for all
  $0<\rx\le \eta/(k+1)$, $x\gg_{k,\eta} 1$, $n<\rvf(x)$, and
  $T\ge 0$, letting $c=(1-\rx)/k$,
  \begin{align*}
    \pr{
      E_{n,k,x}
      \AND
      S_{n:k} > (1-\rx)x
    }
    \ll_{k,r}
    \frac{n\tail\cdf(x)}{x}
    \Sbr{
      T + \nth{a_n} \sup_{c x \le t \le x + 2h}
      \Iof_{\eta, n,r}(t\AND T)
    }.
  \end{align*}
\end{lemma}
\begin{proof}
  The LHS is increasing in $\rx$ and is 0 if $n<k$.  So it suffices
  to prove the bound for $\rx = \eta/(k+1)$ and $n\ge k$.  If
  $S_{n:k}>(1-\rx) x$, then $X_{n:1}>c x$.  By $\zeta_{n,x} \ge r
  a_n$, if $X_{n:k} > \zeta_{n,x}$, then $\tau_n \ge k$, $S'_n \ge
  S_{n:k}$, and $X_{n:1}$ is the largest $X_i$ greater than $r a_n$.
  Thus, $\{E_{n,k,x}\AND S_{n:k} > (1-\rx) x\} \subset \{S_n\in
  x+I\AND X_{n:1}>c x\AND S'_n > (1-\rx) x\AND \tau_n\ge 1\}$ and so
  \begin{align}
    \pr{E_{n,k,x}\AND S_{n:k} > (1-\rx) x}
    &\le
    \sum_{m=1}^n P_m(x) \pr{\tau_n = m},
    \label{e:large-sum-E-G}
  \end{align}
  where $P_m(x) = \pr{B_{n-m} + U_m \in x+I\AND U_m > (1-\rx) x\AND
    u_{m:1}> c x}$.  Put $J = (-h, h)$ and $z_j = (1-\rx) x + j h$.
  By $I - I = J$,
  \begin{align*}
    P_m(x)
    &=
    \sumzi j
    \pr{B_{n-m}\in x - U_m + I\AND
      U_m \in z_j + I\AND u_{m:1} > c x
    }
    \\
    &\le
    \sumzi j
    \pr{
      B_{n-m} \in x - z_j + J
      \AND U_m \in z_j + I\AND u_{m:1}> c x
    }.
  \end{align*}
  With $n$ and $m$ being fixed for now, let $q_j(x) = \pr{B_{n-m} \in
    x - z_j + J}$.  Then by the independence of $B_{n-m}$ and $\eno u
  m$ and the union-sum inequality,
  \begin{align}
    P_m(x) \le
    m\sumzi j
    q_j(x)
    \pr{U_m \in z_j + I \AND u_1 > c x}.
    \label{e:P-m-1}
  \end{align}
  Let $X\sim\cdf$ be independent of $u_2, \ldots, u_m$ and $Z = u_2 +
  \cdots + u_m$.  Then
  \begin{align*}
    \pr{U_m \in z_j + I\AND u_m > c x}
    &=
    \pr{X + Z\in z_j + I\AND X > c x\gv X> r a_n}.
  \end{align*}
  Denoting $G_j(x) = \pr{X\in (z_j - Z + I)\cap (c x, \infty)}$, from
  $\tail\cdf(r a_n)\gg_r 1/n$, it follows that
  \begin{align*}
    \pr{U_m \in z_j + I\AND u_m > c x}\ll_r n G_j(x).
  \end{align*}
  Then \eqref{e:P-m-1} yields $P_m(x)\ll_r m n \sumzi j q_j(x)
  G_j(x)$.  Let
  \begin{align*}
    D_j(x) = \Sbr{\frac{x G_j(x)}{\tail\cdf(x)} - T}^+.
  \end{align*}
  By $G_j(x) \le \cbfrac{\tail\cdf(x)}{x} (T + D_j(x))$,
  \begin{align} \label{e:large-sum-1}
    P_m(x)
    \ll_r
    m n\cbfrac{\tail\cdf(x)}{x}
    \sumzi j q_j(x) [T+D_j(x)].
  \end{align}
  We need to bound $\sum q_j(x) (T+ D_j(x))$.  First, by
  $x-z_j + J  = \rx x - j h + (-h, h)$, there are at
  most two $j\ge 0$ such that $B_{n-m}\in x - z_j + I$.  Then
  \begin{align} \label{e:large-sum-2}
    \sumzi j q_j(x)
    \le
    \mean\Sbr{
      \sumzi j \cf{B_{n-m}\in x - z_j + J}
    } 
    \le 2.
  \end{align}
  To bound $\sum q_j(x) D_j(x)$, put $\jmath = \Cil{\rx x/h}$.  For
  $j>\jmath$, by $z_j-h\ge x$ and $B_{n-m}\ge 0$, $q_j(x)=0$.  For
  $0\le j \le \jmath$, if $\cdf(r a_n)>0$ and $1\le m<n$, then by 
  $\pr{X_{n-m:1} \le r a_n}\gg_r 1$,
  \begin{align*}
    q_j(x) 
    &=
    \pr{S_{n-m}\in x - z_j + J\AND X_{n-m:1} \le r a_n}/
    \pr{X_{n-m:1} \le r a_n}
    \\
    &\ll_r
    \pr{S_{n-m}\in x - z_j + J\AND X_{n-m:1} \le r a_n}.
  \end{align*}
  Then by Lemma \ref{l:local-ldp},
  \begin{align*}
    q_j(x)
    &\ll_r
    [1/a_{n-m} + 1/(r a_n)]  e^{-(x-z_j-h)/(r a_n) +
      O_r((n-m)/ \rvf(a_n))}
    \\
    &\ll_r
    (1/a_{n-m}) e^{-(x-z_j)/(r a_n)}.
  \end{align*}
  On the other hand, if $m=n$ or $\cdf(r a_n)=0$, then $B_{n-m}=0$ and
  $q_j(x) = \cf{x-z_j\in J}$, and in the case $\cdf(r a_n)=0$, $n\ll_r
  1$.  As a result, the above bound still holds.  Letting
  \begin{align*}
    H(x)=\sum_{0\le j\le \jmath} e^{-(x- z_j)/(r a_n)} D_j(x),
  \end{align*}
  it follows that
  \begin{align}
    \sumzi j q_j(x) D_j(x)
    \ll_r \frac{H(x)}{a_{n-m}}.
    \label{e:large-sum-3}
  \end{align}
  To bound $H(x)$, denote $\Lambda_x(t) = \pr{X\in (t + I)\cap (c x,
    \infty)}$.  By independence of $X$ and $Z$, $G_j(x) =
  \mean[\Lambda_x(z_j - Z)]$.  Then by Jensen's inequality,
  \begin{align*}
    D_j(x) \le \mean\Sbr{\frac{x \Lambda_x(z_j - Z)}{\tail\cdf(x)} - T
    }^+
  \end{align*}
  and hence $H(x)\le \mean V(Z, x)$,  where
  \begin{align*}
    V(s,x) = 
    \sum_{0\le j\le\jmath}
    e^{-(x - z_j)/(r a_n)}
    \Sbr{
      \frac{ x \Lambda_x(z_j-s)}{\tail\cdf(x)} - T
    }^+.
  \end{align*}
  Observe that if $s\ge (1-c) x + 2h$, then for $0\le j \le \jmath$,
  $(z_j - s+I)\cap (c x,\infty)=\emptyset$ and hence $V(s,x)=0$.  Thus
  by $Z\ge 0$, 
  \begin{align} \label{e:large-sum-4}
    H(x)
    \le
    \sup_{0\le s<(1-c)x + 2h} V(s,x)
  \end{align}
  Given $s\ge 0$, put $\imath = \min\{j\ge 0: z_j - s
  + h>c x\}$.  Then for $j<\imath$, $\Lambda_x(z_j - s) =  0$.
  Consequently, for $s\ge 0$,
  \begin{align*}
    V(s,x) = 
    \sum_{\imath\le j\le \jmath}
    e^{-(x - z_j)/(r a_n)}
    \Sbr{
      \frac{ x \Lambda_x(z_j - s)}{\tail\cdf(x)} - T
    }^+.
  \end{align*}
  For $\imath\le j\le \jmath$, one gets $z_j - s \le x - s + h \le x
  + h$ and $z_j - s\ge z_\imath -s > c x - h > x/(2k) - h$.  Then for
  $x\gg_k 1$,
  \begin{align*}
    \Lambda_x(z_j - s)
    \le
    \cdf(z_j - s + I)
    =
    \lfrac{\tail\cdf( z_j - s )\,\omega(z_j - s)}{(z_j - s)}
    \ll_k
    \lfrac{\tail\cdf(x) \,\omega(z_j - s)}{x}.
  \end{align*}
  Meanwhile, for $t\in z_j - s + I$, by $z_j-s + I \subset
  (t-h+I) \cup (t + I)$, 
  \begin{align*}
    \omega(z_j - s)
    &=
    \frac{(z_j - s) \cdf(z_j - s+I)}{\tail\cdf(z_j - s)}
    \\
    &\le
    \frac{(z_j - s) \cdf(t-h+I)}{\tail\cdf(z_j - s)}
    +
    \frac{(z_j - s) \cdf(t+I)}{\tail\cdf(z_j - s)}
    \ll
    \omega(t-h) + \omega(t).
  \end{align*}
  Combining the above three displays yields
  \begin{align*} 
    V(s, x)
    &\ll_k
    \sum_{\imath\le j\le\jmath}
    e^{-(x- z_j)/(r a_n)} [O_k(\omega(z_j - s)) - T]^+
    \\
    &\ll_k
    \sum_{\imath\le j\le\jmath}
    h e^{-(x - z_j)/(r a_n)}
    [\omega(z_j - s) - \Omega_k(T)]^+
    \\
    &\ll_r
    \sum_{\imath\le j\le\jmath}
    \int_{z_j - s}^{z_j - s + h} e^{-(x - s - t)/(r a_n)}
    [\omega(t-h) + \omega(t) - \Omega_k(T)]^+
    \,\dd t
  \end{align*}
  where the last line uses $e^{-(x-z_j)/(r a_n)} \asymp_r e^{-(x- s
    -t)/(r a_n)}$ for $t\in z_j - s + I$.  As a result,
  \begin{align*}
    V(s,x)
    \ll_r
    \int_{z_\imath - s- h}^{x-s+2h} e^{-(x-s+2h-t)/(r a_n)}
    [\omega(t)- \Omega_k(T)]^+\,\dd t.
  \end{align*}
  Recall $0\le x-s + 2h - (z_\imath - s- h) = \rx x - \imath h + 3 h \le
  \rx x + 3h$ and $c x - h \le z_\imath - s$.  By $c\eta > \rx$, for
  $x\gg_{k,\eta} 1$, $\rx x + 3h < \eta (z_\imath - s- h)$, giving
  $z_\imath - s - h > (1-\eta) (x-s + 2h)$, and so $V(s,x) \ll
  \Iof_{\eta, n, r}(x-s+2h, \Omega_k(T))$ by the definition of
  $\Iof_{\eta, n, r}$ in \eqref{e:IOF-def}.  Then by
  \eqref{e:large-sum-4}, 
  \begin{align*}
    H(x)
    &\ll_{k,r}
    \sup_{0\le s < (1-c) x + 2h}
    \Iof_{\eta, n, r}(x-s + 2h\AND \Omega_k(T))
    \\
    &\le_{k,r}
    \sup_{c x \le t \le x+2h} 
    \Iof_{\eta, n, r}(t\AND \Omega_k(T)),
  \end{align*}
  which together with \eqref{e:large-sum-1} -- \eqref{e:large-sum-3}
  yields
  \begin{align*}
    P_m(x)
    &\ll_{k,r}
    m n \cbfrac{\tail\cdf(x)}{x}
    \Sbr{
      T + \nth{a_{n-m}} \sup_{c x \le t \le x+2h}
      \Iof_{\eta, n, r}(t\AND \Omega_k(T))
    }.
  \end{align*}
  Note that the implicit constant in $\Omega_k(T)$ does not depend on
  $n$ or $m$.  Let $\Omega_k(T)\ge CT$.  If $C\ge 1$, then
  $\Iof_{\eta,n,r}(t\AND \Omega_k(T))\le \Iof_{\eta,n,r}(t\AND T)$.
  If $0<C<1$, then since $T\ge 0$ is arbitrary, the above inequality
  still holds if $T$ is replaced with $T/C$.  As a result, 
  \begin{align*}
    P_m(x)
    &\ll_{k,r}
    m n\cbfrac{\tail\cdf(x)}{x}
    \Sbr{
      T/C + \nth{a_{n-m}} \sup_{c x \le t \le x+2h}
      \Iof_{\eta,n,r}(t\AND T)
    }
    \\
    &\ll_{k,r}
    m n \cbfrac{\tail\cdf(x)}{x}
    \Sbr{
      T + \nth{a_{n-m}} \sup_{c x \le t \le x + 2h}
      \Iof_{\eta,n,r}(t\AND T)
    }.
  \end{align*}
  In any case, $\Omega_k(T)$ can be replaced with $T$.  Combining this
  bound with \eqref{e:large-sum-E-G} yields
  \begin{align*}
    &
    \pr{E_{n,k,x}\AND S_{n:k}>(1-\rx) x}
    \\
    &
    \ll_{k,r}
    \frac{n\tail\cdf(x)}{x} 
    \sum_{m=1}^n m
    \Sbr{
      T + \nth{a_{n-m}} \sup_{c x \le t \le x + 2h}
      \Iof_{\eta,n,r}(t\AND T)
    } \pr{\tau_n = m}.
  \end{align*}
  As remarked before Lemma \ref{l:large-sum}, $\pr{\tau_n = m} \ll_r
  \theta^m / m!$ for all $n\ge 1$ and $m\ge 0$, with $\theta\ll_r 1$.
  Therefore, $\sum_{m=1}^n m \pr{\tau_n = m} \ll_r 1$.  On the other
  hand, 
  \begin{align*}
    \sum_{m=1}^n \frac{m}{a_{n-m}} \pr{\tau_n=m}
    &\ll_r
    \sum_{m\le n/2} \frac{m}{a_{n-m}} \frac{\theta^m}{m!}
    + \sum_{m>n/2} \frac{m}{a_0} \frac{\theta^m}{m!}
    \\
    &\ll_r
    \nth{a_n} + \frac{\theta^{n/2}}{(\lceil n/2 \rceil -1)!},
  \end{align*}
  which is still $O_r(1/a_n)$.  This then finishes the proof.
\end{proof}

\subsection{Proof of Theorem \ref{t:sum-small-n}}
We shall give a detailed proof for the case $\xpp\in (0,1)$ and only
sketch a proof for the case $\xpp=1$ at the end, which is similar
and actually simpler.  Let $X^*, \inum{X^*}$ be i.i.d.\
$\sim \cdf^*(x) = \pr{0<X\le x\gv X>0}$.  As $x\toi$,
\begin{align*}
  \pr{X^*>x}
  =
  \lfrac{\tail\cdf(x)}{\xpp} \sim \lfrac{1}{\rvf^*(x)}
  \quad
  \text{with}\ \rvf^*(x) = \xpp \rvf(x).
\end{align*}
Whatever only depends on $\{\cdf^*, \rvf^*, h\}$ can also be treated
as only depending on $\{\cdf, \rvf, h\}$.  Objects defined via $X^*$
are marked by $*$, e.g., $\omega^*(x)$ and $E^*_{n,k,x}$.  Note that 
$\zeta^*_{n,x} = (a^*_n)^{1-\gamma} x^\gamma$, with the same $\gamma$
as in $\zeta_{n,x}$.  Let $Z, \inum Z$ be i.i.d.\ following the
distribution of $-X$ conditioning on $X\le 0$, and also be independent
from $(X^*, \inum{X^*})$.  Let
\begin{align*}
  S^*_n = \cum {X^*}n,
  \quad
  W_n = \cum Zn.
\end{align*}
Then for $m\le n$, conditioning on $N_n = m$, $(S^+_n, S^-_n)\sim
(S^*_m\AND  W_{n-m})$.

Fix $p = \xpp/2$ and $\ell\in\RV_0$ such that $\ell(x)/ \ln x\toi$
as $x\toi$.  Given $x>0$, denote
\begin{align*}
  x_n = x + S^-_n, \quad z_{n-m} = x + W_{n-m}.
\end{align*}
For each $n\ge 1$, define $b_n=0$ if $n<\ell(x)$, and $b_n= p n$ if
$n\ge \ell(x)$.  Then
\begin{align} \label{e:sum-small-n1}
  \cdf\cvp n(x+I)
  \le
  \pr{N_n < b_n} +
  \pr{S_n\in x + I\AND N_n\ge b_n}.
\end{align}
Since $\{S_n \in x + I\} = \{S^+_n \in x_n + I\}$, 
\begin{align} \label{e:sum-small-n2}
  \begin{split}
    \pr{S_n\in x + I\AND N_n \ge b_n}
    &=
    \sum_{b_n\le m\le n}
    \pr{S^+_n \in x_n + I\gv N_n=m} \pr{N_n=m}
    \\
    &=
    \sum_{b_n\le m\le n} \pr{S^*_m \in z_{n-m} + I}
    \pr{N_n=m}.
  \end{split}
\end{align}

Let $L(x)>0$, $T\ge 0$, $\eta\in (0,1)$, $r\in (0,1]$, and $1/2\le c_1
< c_0 \le 1\le c_2$ be fixed.  Fix $\rx\in (0,1)$, such that
\begin{align} \label{e:SRT-epsilon}
  \rx \le \eta/(\kappa+1), \quad (1-\rx) c_0\ge c_1.
\end{align}
Define
\begin{align*}
  \Delta_{m,x}
  &=
  \pr{
    S^*_m \in x + I\AND X^*_{m:\kappa+1} > \zeta^*_{m, x}
  },
  \\
  \Delta_{m,k,x}
  &=
  \pr{
    E^*_{m,k,x}\AND S^*_{m:k} \le (1-\rx) x
  },
  \\
  D_m(x)
  &= 
  \sum_{1\le k\le\kappa} \pr{
    E^*_{m,k,x}\AND S^*_{m:k} >
    (1-\rx) x
  }.
\end{align*}
For each $b_n\le m\le n$, by $z_{n-m}\ge x$ and independence of
$S^*_m$ and $z_{n-m}$,
\begin{align} \label{e:sum-small-n3}
  \begin{split}
    &
    \pr{S^*_m \in z_{n-m} + I}
    \\
    &\le
    \sup_{z\ge x\AND b_n\le m\le n} \Delta_{m,z}
    + \sum_{0\le k\le \kappa}
    \sup_{z\ge x\AND m\le n} \Delta_{m,k,z}
    + \mean D_m(z_{n-m}).
  \end{split}
\end{align}
Fix $c\in (0,1)$, such that $c a^*_n \le a_n$ for all $n\ge 0$.  By
Lemma \ref{l:large-sum}, for $x\gg_\eta 1$ and $m<\rvf^*(x)$,
\begin{align*}
  D_m(x)
  &\ll_r
  \frac{m\tail\cdf^*(x)}{x}
  \Sbr{
    T + \nth{a^*_m} \sup_{(1-\rx) x/\kappa \le t \le x+2h}
    \Iof^*_{\eta,m, c r}(t, T)
  }.
\end{align*}
For all $0<\delta\le p^{1/\alpha}$ and $x\gg 1$, $\rvf(\delta x) \le
\rvf^*(x)$.  Then for $m\le n<\rvf(\delta x)$, as $m<\rvf^*(x)$ and
$z_{n-m}\ge x$, the above inequality yields
\begin{align*}
  \mean D_m(z_{n-m})
  &
  \ll
  \mean\Cbr{
    \frac{m \tail\cdf^*(z_{n-m})}{z_{n-m}}
    \Sbr{
      T + \nth{a^*_m}
      \sup_{(1-\rx) z_{n-m}/\kappa \le t \le z_{n-m}+2h}
      \Iof^*_{\eta,m, c r}(t, T)
    }
  }.
\end{align*}
Since $\omega^*(x) = \omega(x)$ and $c a^*_m \le a_m$, from
\eqref{e:IOF-def}, $\Iof^*_{\eta,m,c r}(x,T) \le
\Iof_{\eta,m,r}(x,T)$.  It follows that
\begin{align*}
  \mean D_m(z_{n-m}) 
  &\ll
  \mean\Cbr{
    \frac{m\tail\cdf(x_n)}{x_n}
    \Sbr{
      T + 
      \nth{a_m} \sup_{(1-\rx) x_n/\kappa \le t \le x_n+2h}
      \Iof_{\eta,m,r}(t, T)
    }
    \,\vline\, N_n=m
  }.
\end{align*}
Combine the above inequality with
\eqref{e:sum-small-n1} -- \eqref{e:sum-small-n3} to get
\begin{align*}
  \cdf\cvp n(x+I)
  &\ll
  \pr{N_n<b_n}  + 
  \sup_{z\ge x\AND b_n\le m\le n} \Delta_{m,z}
  + \sum_{0\le k\le \kappa}
  \sup_{z\ge x\AND m\le n}
  \Delta_{m,k,z}
  \\
  &\quad\
  + \mean\Cbr{
    \frac{N_n\tail\cdf(x_n)}{x_n}
    \Sbr{
      T+
      \nth{a_{N_n}} \sup_{(1-\rx) x_n/\kappa \le t \le x_n+ 2h}
      \Iof_{\eta, N_n, r}(t, T)
    }
  }.
\end{align*}
For $1\ll_\eta c_0 x\le y\le c_2 x$, the inequality still holds if
$(x, x_n)$ is replaced with $(y, y_n)$, where $y_n = y+S_n$.  Then by
$c_0 x_n\le y_n \le c_2 x_n$ and $\tail\cdf(y_n)/y_n \ll
\tail\cdf(x_n)/x_n$, it is seen
\begin{align*}
  \sup_{c_0 x\le y\le c_2 x} \cdf\cvp n(y+I)
  &\ll
  \pr{N_n< b_n} + 
  \sup_{z\ge c_1 x\AND b_n\le m\le n}
  \Delta_{m,z}
  + \sum_{0\le k\le\kappa}
  \sup_{z\ge c_1 x\AND m\le n}
  \Delta_{m,k,z}
  \\
  &\quad\
  + \mean\Cbr{
    \frac{N_n\tail\cdf(x_n)}{x_n}
    \Sbr{
      T+
      \nth{a_{N_n}} \sup_{c_1 x_n/\kappa \le t \le c_2 x_n + 2h}
      \Iof_{\eta, N_n, r}(t, T)
    }
  }.
\end{align*}
Taking sum over $L(x)\le n<\rvf(\delta x)$, whether or not
$L(x)<\ell(x)$, one gets 
\begin{align} \label{e:sum-small-n4}
  G_\delta(x)
  :=
  \sum_{L(x)\le n<\rvf(\delta x)}
  \sup_{c_0 x\le y\le c_2 x} \cdf\cvp n(y + I)
  \ll \sum_{i=0}^3 Q_i 
\end{align}
where
\begin{gather*}
  Q_0
  =
  \sum_{n\ge \ell(x)} \pr{N_n<p n},
  \\
  Q_1
  =
  \sum_{n < \ell(x)} \sup_{z\ge c_1 x\AND m\le n}
  \pr{
    X^*_{m:\kappa+1} > \zeta^*_{m, z}
  }
  + \sum_{n<\rvf(\delta x)}
  \sup_{z\ge c_1 x\AND p n\le m\le n} \Delta_{m,z},
  \\
  Q_2
  =
  \sum_{n<\rvf(\delta x)} \sum_{0\le k\le \kappa}
  \sup_{z\ge c_1 x\AND m\le n} \Delta_{m,k,z},
\end{gather*}
and
\begin{align*}
  Q_3
  &=
  \sum_{L(x)\le n<\rvf(\delta x)} \mean\Sbr{
    \frac{T n \tail\cdf(x_n)}{x_n}
    +
    \frac{N_n} {a_{N_n}}
    \frac{\tail\cdf(x_n)}{x_n}
    \sup_{c_1 x_n/\kappa \le t \le c_2 x_n + 2h}
    \Iof_{\eta, N_n, r}(t, T)
  }.
\end{align*}

By $p=\xpp/2$, $\pr{N_n < p n} \ll e^{-c n}$ for some $c>0$.  By
$\ell(x)/\ln x\toi$,
\begin{align*}
  Q_0\ll
  e^{-c \ell(x)}/(1-e^{-c}) = o(x^{-M})
\end{align*}
for any $M>0$.  On the other hand, apply Lemma
\ref{l:many-large-jumps} to $Q_1$ and Lemma \ref{l:small-sum} to
$Q_2$.  It follows that, for $0< \delta\ll_\eta 1$,
\begin{align*} 
  \Lsup_{x\toi}
  \frac{x G_\delta(x)}{\rvf(x)}
  &\ll
  o(\delta^\alpha)
  +
  T\Lsup_{x\toi}
  \frac{x}{\rvf(x)} \sum_{L(x)\le n < \rvf(\delta x)}
  n\mean\Sbr{
    \frac{\tail\cdf(x_n)}{x_n}
  }
  +
  \Lsup_{x\toi} \XR(x, \delta).
\end{align*}
Since for $z\ge x\ge
1$, $\tail\cdf(z)/z \ll \tail\cdf(x)/x$,
\begin{align*}
  \frac{x}{\rvf(x)}\sum_{L(x)\le n<\rvf(\delta x)}
  n\mean\Sbr{\frac{\tail\cdf(x_n)}{x_n}}
  &\ll
  \tail\cdf(x)^2
  \sum_{L(x)\le n<\rvf(\delta x)} n
  \\
  &\ll
  \tail\cdf(x)^2 \rvf(\delta x)^2 \ll \delta^{2\alpha}.
\end{align*}
This combined the previous display then yields the desired result.

Finally, we comment on the case $\xpp = 1$.  For $n\ge
1$, $x>0$, and $\rx\in (0,1)$,
\begin{align*}
  \cdf\cvp n(x+I)
  &\le
  \pr{S_n \in x + I\AND X_{n:\kappa+1} > \zeta_{n,x}}
  + 
  \sum_{0\le k\le\kappa} \pr{
    E_{n,k,x}\AND S_{n:k} \le (1-\rx) x
  }
  \\
  &\quad
  +
  \sum_{1\le k\le \kappa} \pr{
    E_{n,k,x}\AND S_{n:k} >
    (1-\rx) x
  }.
\end{align*}
Then \eqref{e:sum-small-n4} can be simplified into
\begin{align*} 
  \begin{split}
    G_\delta(x)
    &\ll
    \sum_{n<\rvf(\delta x)}
    \sup_{z\ge x/2}
    \pr{
      S_n \in z + I\AND X_{n:\kappa+1} > \zeta_{n,z}
    } 
    \\
    &\quad\
    + \sum_{n<\rvf(\delta x)} \sum_{0\le k\le \kappa}
    \sup_{z\ge x/2}
    \pr{
      E_{n,k,z}\AND S_{n:k} \le (1-\rx) z
    }
    \\
    &\quad\
    + 
    \frac{\tail\cdf(x)}{x}
    \sum_{1\le k\le \kappa}
    \sum_{L(x)\le n<\rvf(\delta x)} \Sbr{
      T n
      +
      \frac{n} {a_n}
      \sup_{c_1 x/\kappa \le t \le c_2 x + 2h}
      \Iof_{\eta,n,r}(t, T)
    }.
  \end{split}
\end{align*}
By following the rest of the proof for the case $\xpp\in (0,1)$, the
desired result obtains.

\subsection{Proof of Proposition \ref{p:sum-small-n-low}}
We need the following result.  Although it bears some similarity with
Lemma \ref{l:local-ldp}, it is more appropriate to regard it as a
variation of the LLT.
\begin{lemma} \label{l:local-lower-bound}
  Let \eqref{e:tail-RV} hold and $\alpha \in (0,1)$.  Then there is
  $C>0$ only depending on $\{\cdf,\rvf\}$, such that given
  $\eta >0$, for all $n\gg_\eta 1$ and $s\ge a_n$,
  \begin{align*}
    \sup_x
    |a_n \pr{S_n \in a_n x + I\AND X_{n:1}\le s} - h \limdf(x)|
    \le
    C h n/\rvf(s)+\eta,
  \end{align*}
  where $h>0$ arbitrary is $\cdf$ is non-lattice, and the span of
  $\cdf$ otherwise.
\end{lemma}
\begin{proof}
  We shall only prove the lemma for the non-lattice case by modifying
  the argument in \cite{stone:65:ams}.  The lattice case can be proved
  similarly based on \cite{gnedenko:54:aw}, p.~236-240.  Define
  $K(x) = \sin(x/2)/(\pi x)$ and $K_\sigma(x) = \sigma^{-1}
  K(x/\sigma)$ for $\sigma>0$.  As in \cite{stone:65:ams}, define
  \begin{align*}
    V_n(x,r,\sigma)
    =
    \int K_\sigma(x-y) \pr{S_n \in a_n y + (0, a_n r]}\, \dd y.
  \end{align*}
  Meanwhile, define similarly
  \begin{align*}
    V_{n,s}(x,r,\sigma)
    =
    \int K_\sigma(x-y) \pr{S_n \in a_n y + (0, a_n r]\AND
    X_{n:1}\le s}\, \dd y.
  \end{align*}
  In the following, $C$ is a constant only depending on $\{\cdf,
  \rvf\}$ that may change from line to line.  It suffices to show
  that given $\eta>0$ and $D\ge 1$, for $n\gg_{\eta,D} 1$, $s\ge a_n$,
  $r>0$ and $\sigma \ge (D a_n)^{-1}$,
  \begin{align}  \label{e:local-low-main}
    \sup_x |V_n(x,r,\sigma) - V_{n,s}(x,r,\sigma)|
    \le r[C n/\rvf(s) + \eta].
  \end{align}
  Indeed, the argument that starts with the last two inequalities on
  p.~550 of \cite{stone:65:ams} can be applied to $V_{n,s}(x,r,
  \sigma)$ and $\pr{S_n\in a_n y + (0, a_n r]\AND X_{n:1}\le s}$, to
  get inequalities similar to those for $V_n(x,r,\sigma)$ and
  $\pr{S_n\in a_n y + (0, a_n r]}$ provided on p.~551--551 of
  \cite{stone:65:ams}.  These inequalities combined with
  \eqref{e:local-low-main} then yield the lemma.  To show
  \eqref{e:local-low-main}, let $f(t) = \mean[e^{\iunit t X}]$,
  $f_s(t) = \mean[e^{\iunit t X} \cf{X\le 
    s}]$, and $\lambda_{r,\sigma}(t) = (1-|\sigma t|)^+ (1-e^{-\iunit
    r t})/(\iunit r t)$.  Then from \cite{stone:65:ams},
  $|\lambda_{r,\sigma}(t)| \le 1$ and given $\rx>0$,
  \begin{align*}
    V_{n,s}(x,r,\sigma)
    &=
    \frac{r}{2\pi}
    \int_{-1/\sigma}^{1/\sigma} e^{-\iunit x t}
    \lambda_{r,\sigma}(t) f_s(t/a_n)^n\,\dd t
    \\
    &=
    \frac{r}{2\pi}
    \Grp{
      \int_{|t| \le (\rx a_n)\wedge (1/\sigma)}+
      \int_{(\rx a_n)\wedge (1/\sigma)< |t|\le 1/\sigma}
    }
    =
    \frac{r}{2\pi}(I_{1,s} + I_{2,s})
  \end{align*}
  and $V_n(x,r,\sigma) = (r/2\pi) (I_1 + I_2)$, where $I_i$ are
  integrals defined likewise in terms of $f$.  Since $f(t)\to 1$ as
  $t\to 0$ and
  \begin{align} \label{e:basic-f}
    \sup |f_s - f| \le \tail\cdf(s),
  \end{align}
  for $|t|\ll 1$, $|f_s(t) - f(t)| \le
  2\tail\cdf(s)|f(t)|$, and so by $|(1+z)^n-1| \le (1+|z|)^n-1$,
  \begin{align*}
    |f_s(t)^n - f(t)^n|
    &\le
    |f(t)|^n \Sbr{
      \Grp{
        1+\Abs{
          \frac{f_s(t)}{f(t)}-1
        }
      }^n-1
    } 
    \\
    &\le
    |f(t)|^n \Sbr{
      \Grp{
        1+2\tail\cdf(s)
      }^n-1
    }
    \\
    &\le
    [e^{2\tail\cdf(s) n}-1] |f(t)|^n,
  \end{align*}
  which yields that, for $n\ge 1$ and $s\ge a_n$, $|f_s(t)^n - f(t)^n|
  \le (C n/\rvf(s)) |f(t)|^n$.  It follows that given $0<\rx\ll 1$,
  for all $x$,
  \begin{align*}
    |I_{1,s} - I_1|
    &\le
    \int_{|t|<\rx a_n} |f_s(t/a_n)^n - f(t/a_n)^n|\,\dd t
    \le
    \frac{C n}{\rvf(s)} \int_{|t|<\rx a_n} |f(t/a_n)|^n \, \dd t.
  \end{align*}
  Then from \cite{stone:65:ams}, p.~549, $|I_{1,s}-I_1| \le C
  n/\rvf(s)$.  On the other hand, since $\cdf$ is non-lattice,
  $\sup_{\rx\le |t|\le D} |f(t)|<1$.  By \eqref{e:basic-f}, for
  $s\gg_{\rx,D} 1$, $\sup_{\rx\le |t|\le D} |f_s(t)|<1$.  Then as
  \cite{stone:65:ams}, p.~549, for $n\gg_{\rx,D,\eta} 1$, $s\ge
  a_n$, and $\sigma\ge (D a_n)^{-1}$, $|I_{2,s} - I_2|\le |I_{2,s}| +
  |I_2| \le \eta$ for all $x$.  Finally, since
  \begin{align*}
    0
    \le
    V_n(x,r,\sigma) - V_{n,s}(x,r,\sigma)
    \le
    \frac{r}{2\pi}(|I_{1,s} - I_1| + |I_{2,s} - I_2 |),
  \end{align*}
  combining the above bounds, \eqref{e:local-low-main} follows.
\end{proof}

\begin{proof}[Proof of Proposition \ref{p:sum-small-n-low}]
  Let $E=[\theta_1, \theta_2]$.  Fix $G=(c,d)\supset E$, such that
  $\inf_G \limdf > 0$.  Denote $c_n = c a_n$ and $d_n = d a_n$.  Then
  for $x\gg_G 1$, $0<\delta \ll_{G,\gamma} 1$, and
  $n<\rvf(\delta x)$, $x - d_n > \zeta_{n,x}$, and hence the events
  $\{X_i> x - d_n\AND X_j\le \zeta_{n,x}\AND j\ne i\}$, $i=1,\ldots,
  n$, are disjoint.  It follows that
  \begin{align*}
    \cdf\cvp n(x+J)
    &\ge
    n \sum_{c_n < j h < d_n}
    \pr{X_n \in x - j h + I\AND S_{n-1}\in j h + I\AND X_{n-1:1}
      \le \zeta_{n,x}}
    \\
    &=
    n \sum_{c_n < j h < d_n}
    \cdf(x-j h+I) \pr{S_{n-1}\in j h + I\AND X_{n-1:1}\le
      \zeta_{n,x}}.
  \end{align*}
  In the following, let $x\gg_G 1$, $0<\delta \ll_{G,\gamma} 1$, and
  $n_0\gg_{G,\gamma} 1$, which at each step of argument may be further
  increased.  By Lemma \ref{l:local-lower-bound}, for $n_0\le n
  \le\rvf(\delta x)$, and $z\in (c_n, d_n)$,
  \begin{align*}
    \pr{S_{n-1}\in z + I\AND X_{n-1:1}\le \zeta_{n,x}} 
    \ge \frac{0.9 h}{a_n} \limdf(z/a_n).
  \end{align*}
  Then
  \begin{align*}
    \cdf\cvp n(x+J)
    &\ge
    \frac{0.9 h n}{a_n}
    \sum_{c_n < j h < d_n} \limdf(j h/a_n) \cdf(x - j h + I).
  \end{align*}
  Because $\limdf\in C^\infty$, we also have
  \begin{align*}
    \cdf\cvp n(x+J)
    &\ge
    \frac{0.9 h n}{a_n}
    \sum_{c_n < j h < d_n}
    \limdf((j-1) h/a_n) \cdf(x - j h + I)
    \\
    &=
    \frac{0.9 h n}{a_n}
    \sum_{c_n -h < j h < d_n - h }
    \limdf(j h/a_n) \cdf(x - j h - h+ I).
  \end{align*}
  Take average of the inequalities to get
  \begin{align*}
    \cdf\cvp n(x+J)
    &\ge
    \frac{0.9 h n}{2 a_n}
    \sum_{c_n < j h < d_n-h}
    \limdf(j h/a_n) [\cdf(x - j h + I)+\cdf(x - j h - h+ I)].
    \end{align*}
    For each $j$ in the sum and $t\in [j h, j h+h)$, $\cdf(x-t+I) \le
    \cdf(x-j h + I) + \cdf(x - j h -h+I)$ and $\limdf(j h/a_n)\ge 0.9
    \limdf(t/a_n)$.  Then
    \begin{align*}
      \cdf\cvp n(x+J)
      &\ge
      \frac{0.4 n}{a_n} \int_{\theta_1 a_n}^{\theta_2 a_n}
      \limdf(t/a_n) \cdf(x-t+I)\,\dd t
      \\
      &=
      \frac{0.4 n}{a_n} \int_{\theta_1 a_n}^{\theta_2 a_n}
      \limdf(t/a_n) \frac{\tail\cdf(x-t)\omega(x-t)}{(x-t)}
      \,\dd t
      \\
      &\ge
      \frac{0.3 n}{a_n} \nth{x\rvf(x)} 
      \int_{\theta_1 a_n}^{\theta_2 a_n} \limdf(t/a_n) 
      \omega(x-t)\,\dd t.
    \end{align*}
    The claim of the proposition then easily follows.
\end{proof}

\section{Proofs of Propositions} \label{s:Prior}
\subsection{Proof of Proposition \ref{p:L-prior}}
Without loss of generality, assume $M(x)$ is continuous.  First, note
that 
\begin{align} \label{e:M-bound}
  M(x)/x\ll  1, \quad x\ge 1.
\end{align}
Indeed, for $x\gg 1$, since $(x,2x]$ can be divided into $\lceil
x/h\rceil$ disjoint intervals of length at most $h$, there is $t\in
(x,2x]$ with $\cdf(t+I) \ge \cdf((x,2x])/\lceil x/h\rceil \asymp
\tail\cdf(x)/x\asymp \tail\cdf(t)/t$, yielding $\omega(t) \gg 1$.
Since $x/M(x)\in\RV$, then $x/M(x)\gg t/M(t)\gg \omega(t)$, yielding
\eqref{e:M-bound}.

The proofs of 1) and 2) have a large overlap.  Given $L\in\RV$ as
prescribed, let $b(x) = L(x)/g(x)$ and $V(x) = \rvf(x) M(x)$.  Then
$V\in\RV_{\alpha+\beta}$ and
\begin{align} \label{e:M-b-V}
  b(x) V(x) =\sqrt{x M(x)} L(x)
\end{align}
Since $L(x)\ll g(x)/(\ln x)^\gamma$ in 1) and $L(x)\ll g(x)^\rx$ with
$\rx<1$ in 2), and $g(x)\toi$, then $b(x)\to 0$.  Since $V$ is
strictly increasing and maps $(0,\infty)$ on to $(0,\infty)$,
for every $x>0$, there is a unique $T(x)>0$, such that $V(T(x)) = b(x)
V(x)$.  As $x\toi$, it is seen $T(x) = o(x)$ and by \eqref{e:M-b-V},
$T(x)\toi$.  Then by Potter's theorem \cite[][Th.~1.5.6]{bingham:89},
for any $r>\alpha+\beta$,
\begin{align} \label{e:prior-T}
  T(x)/x = O(b(x)^{1/r}).
\end{align}
Fix $\theta>0$.  Let $J=(-h, h)$.  For each $n\ge 1$ and $t$, by
$I-I=J$,
\begin{align*}
  \pr{S_n\in t+I\AND X_{n:1}>T(x)}
  &\le
  n \pr{S_n\in t+I\AND X_n > T(x)}
  \\
  &=
  n \sumzi k \pr{S_n \in t+I\AND X_n \in T(x)+k h + I} 
  \\
  &\le
  n \sumzi k\pr{S_{n-1}\in t-T(x)-k h + J} \cdf(T(x) + k h+I)
  \\
  &\ll
  n\sup_{s\ge T(x)} \cdf(s+I),
\end{align*}
with the last line due to \eqref{e:large-sum-2}.  By
\eqref{e:tail-M}, $\cdf(s+I) = \omega(s) \tail\cdf(s)/s \ll
1/V(s)$.  Then for large $x>0$ and $s\ge T(x)$, $\cdf(s+I) \ll
1/V(T(x))$, giving
\begin{align*} 
  \frac{x}{\rvf(x)}
  \sum_{n<L(x)} \sup_{t\in\Reals}
  \pr{S_n\in t+I\AND X_{n:1}>T(x)}
  &\ll
  \frac{x}{\rvf(x) V(T(x))}
  \sum_{n<L(x)} n
  \\
  &\ll
  \frac{x L(x)^2}{\rvf(x)V(T(x))} = b(x).
\end{align*}
For 1), from Lemma \ref{l:local-ldp}, for some $c\ge 0$ only depending
on $\{\cdf, \rvf\}$,
\begin{align*}
  \frac{x}{\rvf(x)} \sum_{n<L(x)} 
  \sup_{t\ge \theta x}
  \pr{S_n\in t+I\AND X_{n:1}\le T(x)}
  &\ll
  x\tail\cdf(x) L(x) e^{-\theta x/T(x)+ c L(x)/\rvf(T(x))}.
\end{align*}
Since $b(x)= o((\ln x)^{-\gamma})$ with $\gamma>\alpha+\beta$, by
\eqref{e:prior-T}, $T(x) = o(x/\ln x)$.  On the other hand, by
\eqref{e:M-b-V}, $L(x)/\rvf(T(x)) = M(T(x))/\sqrt{x M(x)}$, which is
$o(1)$ due to $T(x)=o(x)$ and \eqref{e:M-bound}.  Then the RHS in the
display is $o(1)$, which together with the previous display yields
\eqref{e:low-cut}.  For 2), it suffices to show $L(x) T(x) = o(x)$, as
then the LHS in the display is 0.  Since $L(x) = g(x)^\rx$, 
where $\rx\in (0, 1/(1+\alpha + \beta))$, $b(x) = g(x)^{\rx - 1}$.
Fix $\eta>\alpha+\beta$ such that $(1+\eta)\rx < 1$.  Then by \eqref
{e:prior-T}, $L(x) T(x)/x \ll g(x)^{\rx -(1-\rx)/\eta} = o(1)$,
as desired.

\subsection{Proof of Proposition \ref{p:density}}
Fix $L(x)\equiv 1$, $\eta\in (0,1)$, $r=1$, and $1/2\le c_1<1\le c_2$.
By Theorem \ref{t:sum-small-n} and the proof of Theorem \ref
{t:SRT-alpha-LE-1/2}, it suffices to show $\Lsup_{x\toi}
\XR(x,\delta)=o(1)$ as $\delta\to 0$.  Let $\sigma_{n,x} = a_n^{1-s}
x^s$.  For $n< \rvf(\delta x)$, since $\sigma_{n,x} \ge x^s\vee a_n$
and $E_T$ has density $O(x^{-c})$ at scale $x^s$,
\begin{align*}
  \Iof_{\eta,n,1}(x,T)
  &\ll
  x^c \int_{(1-\eta) x}^x e^{-(x-y)/a_n}
  \cf{\omega(y)>T}\,\dd y
  \\
  &\le
  x^c \sum_{0\le k \le \eta x/\sigma_{n,x}}
  e^{-k\sigma_{n,x}/a_n}
  |E_T\cap [x - (k+1)\sigma_{n,x}, x- k\sigma_{n,x}]|
  \\
  &\ll
  x^c\sum_{0\le k \le \eta x/\sigma_{n,x}}
  e^{-k} x^{-c} \sigma_{n,x}
  \le
  x^s a_n^{1-s}.
\end{align*}
Then $\cbfrac{\tail\cdf(x)}{x} \Iof_{\eta,n,1}(x,T)\ll a_n^{1-s}
f(x)$, where $f(x)=x^{s-1}/\rvf(x)\in\RV$ with exponent
$s-1-\alpha < 0$.  It follows that for $z\ge x$ and $n\ge 1$,
\begin{align*}
  \cbfrac{\tail\cdf(z)}{z}
  \sup_{c_1 z\le t \le c_2 z} \Iof_{\eta, n,1}(t,T)
  \ll
  a_n^{1-s} f(z)
  \ll
  a_n^{1-s} f(x).
\end{align*}
As a result,
\begin{align*}
  \XR(x,\delta)
  \ll
  \frac{x f(x)}{\rvf(x)}
  \sum_{n<\rvf(\delta x)} \mean[N_n/a_{N_n}^s].
\end{align*}
Using the argument in the proof of Lemma \ref{l:R-bound} and noting
$s<2\alpha$,
\begin{align*}
  \sum_{n<\rvf(\delta x)} \mean[N_n/a_{N_n}^s]
  \asymp
  \int_1^{\delta x} \cbfrac{\rvf(t)^2}{t^{1+s}}\,\dd t
  \asymp \frac{\rvf(\delta x)^2}{(\delta x)^s}. 
\end{align*}
It follows that $\XR(x,\delta) \ll \delta^{2\alpha -s}$.  Since
$s<2\alpha$, the desired result follows.

\subsection{Proof of Proposition \ref{p:L-prior-ladder}}
Since $\rvf_+\in \RV_{\alpha \ppar}$ and $x/M(x)=O(x^{2c\alpha \ppar})
= o(\rvf_+(x)^2)$, by Proposition \ref{p:L-prior}, it suffices to show
$\omega_+(x) := x\cdf\!_+(x + I)/\tail\cdf\!_+(x) \ll x/M(x)$ for
$x\ge 1$.  By \eqref{e:renewal-ladder},
\begin{align*}
  \cdf\!_+(x + I)
  &=
  \intzi \frac{\omega(x+y)\tail\cdf(x+y)}{x+y}\,\renf_-(\dd y)
  \\
  &\le
  \intzi \frac{\tail\cdf(x+y)}{M(x+y)}\,\renf_-(\dd y)
  \le
  \nth{M(x)}
  \intzi \tail\cdf(x+y)\,\renf_-(\dd y).
\end{align*}
Since the last integral is equal to $\tail\cdf\!_+(x)/M(x)$, then
the result follows.

\begin{small}
\def\cprime{$'$} \def\cprime{$'$} \def\cprime{$'$}

\end{small}

\end{document}